\documentclass[reqno,12pt]{amsart}
\usepackage{amsmath, amsfonts, amsthm, mathtools, extarrows,amssymb,mathrsfs,dsfont}
\usepackage[x11names]{xcolor}
\usepackage[colorlinks=true, linkcolor=Brown4, citecolor=Chocolate4,pagebackref]{hyperref}
\usepackage{amsrefs}
\usepackage{cleveref}
\usepackage{tikz-cd}
\usepackage[shortlabels]{enumitem}
\setlist[enumerate]{leftmargin=*}
\usepackage[margin=1in,a4paper]{geometry}

\usepackage[indent]{parskip}
\usepackage{verbatim}
\newenvironment{acknowledgements}{\section*{Acknowledgements}}{}
\newtheorem{theorem}{Theorem}[section] 
\newtheorem{lemma}[theorem]{Lemma}     
\newtheorem{corollary}[theorem]{Corollary}
\newtheorem{proposition}[theorem]{Proposition}

\theoremstyle{definition}
\newtheorem{definition}[theorem]{Definition}
\newtheorem{remark}[theorem]{Remark}
\newtheorem{example}[theorem]{Example}
\numberwithin{equation}{section}

\usepackage[T1]{fontenc}

\usepackage{xcolor}
\usepackage{rotating}
\usepackage{appendix}

\definecolor{hypercolor}{rgb}{0.5,0,0.5}
\definecolor{MK_One_Five}{RGB}{90,180,172}

\setitemize[1]{itemsep=5pt}
\setenumerate[1]{itemsep=5pt}

\newcommand\cC{\mathcal{C}}

\newcommand\cH{\mathcal{H}}

\newcommand\cL{\mathcal{L}}

\newcommand\cO{\mathcal{O}}

\newcommand\cS{\mathcal{S}}
\newcommand\cT{\mathcal{T}}

\renewcommand\AA{\mathbb{A}}

\newcommand\CC{\mathbb{C}}

\newcommand\HH{\mathbb{H}}

\newcommand\PP{\mathbb{P}}
\newcommand\QQ{\mathbb{Q}}
\newcommand\RR{\mathbb{R}}

\newcommand\ZZ{\mathbb{Z}}

\newcommand\bD{\mathbf{D}}
\newcommand\bE{\mathbf{E}}

\newcommand\bP{\mathbf{P}}

\newcommand\bT{\mathbf{T}}

\newcommand\bY{\mathbf{Y}}

\newcommand\rH{\mathrm{H}}
\newcommand\rI{\mathrm{I}}

\newcommand\rO{\mathrm{O}}
\newcommand\rP{\mathrm{P}}

\newcommand\rT{\mathrm{T}}

\newcommand\rmd{\mathrm{d}}

\newcommand\bfe{\mathbf{e}}

\newcommand\bfg{\mathbf{g}}

\newcommand\bfu{\mathbf{u}}
\newcommand\bfv{\mathbf{v}}

\newcommand\bfx{\mathbf{x}}

\newcommand\fre{\mathfrak{e}}

\newcommand\frn{\mathfrak{n}}

\newcommand\frq{\mathfrak{q}}

\newcommand\tA{\tilde{A}}
\newcommand\tB{\tilde{B}}

\newcommand\btg{\widetilde{\bfg}}

\newcommand\rank{{\rm rank}}

\newcommand\Pic{{\rm Pic}}

\newcommand\Cl{{\rm Cl}}
\newcommand\SL{{\rm SL}}

\newcommand\GL{{\rm GL}}
\newcommand\Mp{{\rm Mp}}

\newcommand\Sh{{\rm Sh}}
\newcommand\Proj{{\rm Proj}}

\newcommand\node{{\rm node}}

\newcommand\sym{{\rm Sym}}
\newcommand\vol{{\rm vol}}

\newcommand{\diag}{ \mathtt{diag}}

\newcommand{\scF}{{\mathscr{F}}}

\newcommand{\scM}{{\mathscr{M}}}

\newcommand{\Cusp}{{\mathrm{Cusp}}}
\newcommand{\Mod}{{\mathrm{Mod}}}
\newcommand{\Sp}{{\mathrm{Sp}}}
\newcommand{\tg}{\tilde{g}}

\newcommand{\sign}{\mathrm{sign}}
\newcommand{\tr}{\mathrm{tr}}
\newcommand{\Gen}{\mathbf{Gen}}

\newcommand{\MU}{\left(\begin{array}{cc}\rI_d & B \\0 & \rI_d\end{array}\right)}
\newcommand{\MJ}{\left(\begin{array}{cc}0 & -\rI_d \\ \rI_d & 0\end{array}\right)}
\newcommand{\MP}{\left(\begin{array}{cc}U & 0 \\0 & (U^{-1})^\top\end{array}\right)}

\DeclarePairedDelimiterX\set[1]\lbrace\rbrace{#1}

\title[Picard group of Baily--Borel compactifications]{The Picard group of the Baily--Borel compactification of the moduli space of quasi-polarized K3 surfaces and generalizations} 

\author[C.Huang, Z.Li, M. K.-H. M\"uller, Z. Ye]{Chenxin Huang, Zhiyuan Li, Manuel K.-H. M\"uller and Zelin Ye}

\address{$^{1}$ School of Mathematical Sciences, Fudan University, Shanghai, China.}
\email{\textcolor{MK_One_Five}{chuang@math.harvard.edu}}

\address{$^{2}$ Shanghai Center for Mathematical Science, Shanghai, China.}
\email{\textcolor{MK_One_Five}{zhiyuan\underline{ }li@fudan.edu.cn}}

\address{$^{3}$ Department of Mathematics, Technical University of Darmstadt, Berlin, Germany.}
\email{\textcolor{MK_One_Five}{mmueller@mathematik.tu-darmstadt.de}}

\address{$^{4}$ Shanghai Center for Mathematical Science, Shanghai, China.}
\email{\textcolor{MK_One_Five}{zlye24@m.fudan.edu.cn}}


\thanks{Zhiyuan Li is supported by NSFC grant (12121001, 12171090 and 12425105) and Shanghai Pilot Program for Basic Research (No. 21TQ00). Zhiyuan is also a member of LMNS.  Manuel K.-H.\ Müller acknowledges support by Deutsche Forschungsgemeinschaft (DFG, German Research Foundation) through the Collaborative Research Centre TRR 326 Geometry and Arithmetic of Uniformized Structures, project number 444845124 as well as support by the LOEWE professorship in Algebra, project number LOEWE/4b//519/05/01.002(0004)/87. Zelin Ye is supported by the NKRD Program of China (No. 2020YFA0713200).}

\makeatletter
\gdef\@journal{}
\makeatother

\begin{document}
\maketitle

\begin{abstract}
In this paper, we investigate the Picard group of the Baily--Borel compactification of orthogonal Shimura varieties. As a key result, we determine the Picard group of the Baily--Borel compactification of the moduli space of quasi-polarized K3 surfaces, proving that it is isomorphic to $\mathbb{Z}$. Notably, this contrasts with the moduli space of smooth curves, where the Picard group exhibits a more complex structure after natural compactification.

Our result follows from a general theorem for orthogonal Shimura varieties: for even lattices $M$ of signature $(2,n)$ with $n > 8$ satisfying specific arithmetic conditions (e.g. K3 type or $p$-elementary), the rational Picard group of $\overline{\operatorname{Sh}}_\Gamma(M)$ with $\Gamma$ containing the stable orthogonal group is $1$-dimensional. The core of our proof lies in constructing an arithmetic obstruction space that governs the extension of Heegner divisors to the generic points of the boundary in orthogonal Shimura varieties. We further establish a connection between this obstruction space and the space of theta series, demonstrating that the obstruction space is maximal under our conditions.
\end{abstract}

\section{Introduction}

\subsection{Picard group of moduli space of quasi-polarized K3 surfaces}The study of the Picard group of moduli spaces stands as a vital and inspiring challenge in moduli theory, rooted in the groundbreaking contributions of Mumford \cite{Mum65} in 1965. Let $\scM_g$ be the moduli space of smooth curves of genus $g\geq 3$ and let $\overline{\scM}_g$ be its Deligne--Mumford compactification.  Harer \cite{Har83} and Arbarello--Cornalba \cite{AC87} showed that 
\begin{equation}
    \Pic(\scM_g)\cong \ZZ~\hbox{and}~\Pic(\overline{\scM}_g)\cong \ZZ^{2+\left[\frac{g}{2}\right]}
\end{equation}
are free abelian groups.

For K3 surfaces, let $\scF_g$ be the (coarse) moduli space of quasi-polarized K3 surfaces of genus $g\geq 2$. By the global Torelli theorem, $\mathscr{F}_g$ is realized as an orthogonal Shimura variety, whose rich geometric structure has been extensively studied. Recent breakthroughs \cites{BLMM17,DFV24} established that its Picard group $\Pic(\mathscr{F}_g) \cong \mathbb{Z}^{r_g}$ is freely generated by linear combinations of Noether--Lefschetz divisors, with rank given by the intricate formula:
\begin{align*}
r_g = \frac{31g+24}{24} & - \frac{1}{4} \left( \frac{2g-2}{2g-3}\right) \left( \frac{g}{2}\right) - \frac{1}{6}\left( \frac{g-1}{4g-5}\right) \\
& - \frac{1}{6}(-1)^{-\left( \frac{g-1}{3}\right)} - \sum_{k=0}^{g-1} \left\{ \frac{k^2}{4g-4} \right\} - \# \left\{ k \mid \frac{k^2}{4g-4} \in \mathbb{Z},\ 0 \leq k \leq g-1 \right\}
\end{align*}
where $\left( \frac{\cdot}{\cdot} \right)$ denotes the Jacobi symbol.  As a locally symmetric variety, $\scF_g$ has a Baily--Borel compactification by the works of Satake, Baily, and Borel, denoted as $\overline{\scF}_g$.  A natural problem is to figure out the Picard group of $\overline{\scF}_g$. Since $\overline{\scF}_g$ is a highly singular variety (with canonical singularities), one expects that the Picard rank of $\overline{\scF}_g$ is relatively small. The main result of this paper establishes that the Picard rank of $\overline{\mathscr{F}}_g$ is surprisingly minimal—in fact, the smallest possible.
\begin{theorem}\label{thm:mainthm}
    $\Pic(\overline{\scF}_g)\cong \ZZ$, which is spanned by some multiple of the extended Hodge line bundle $\overline{\lambda}$.
\end{theorem}
This demonstrates that despite the complexity of $\Pic(\scF_g)$, the Baily-Borel compactification exhibits extreme rigidity in its theory of Cartier divisors.

\subsection{Picard group of orthogonal Shimura varieties}
Our approach establishes a general method for computing Picard groups of Baily--Borel compactifications of orthogonal Shimura varieties. Let $M$ be a lattice of rank $\geq 5$ and $\Sh_\Gamma(M): = \Gamma \backslash \bD$ the associated orthogonal Shimura variety, where $\bD$ is the Hermitian symmetric domain attached to $M$ and $\Gamma \leq \mathrm{O}(M)$ is a congruence subgroup. Denote by $\overline{\Sh}_\Gamma(M)$ or $(\Gamma\backslash \bD)^{\rm BB}$ its Baily--Borel compactification. In most situations, we may assume $\Gamma$ is contained in the stable orthogonal group $\widetilde{\rO}(M)$. 

According to \cite[Thm. 1.5]{BLMM17}, the rational Picard group $\operatorname{Pic}_{\mathbb{Q}}(\Sh_\Gamma(M))$ is spanned by irreducible components of Heegner divisors on $\Sh_\Gamma(M)$. We define the Heegner subspace $\operatorname{Pic}_{\mathbb{Q}}(\Sh_\Gamma(M))^{\mathrm{Heegner}}$ as the $\mathbb{Q}$-subspace generated by these divisors, and introduce the subgroup
\[
\operatorname{Pic}(\overline{\Sh}_\Gamma(M))^{\mathrm{Heegner}} \subseteq \operatorname{Pic}(\overline{\Sh}_\Gamma(M))
\]
consisting of Cartier divisors whose restrictions to $X$ lie in $\operatorname{Pic}_{\mathbb{Q}}(X)^{\mathrm{Heegner}}$.

In this paper, we investigate the problem when elements of $\operatorname{Pic}_\mathbb{Q}(\overline{\Sh}_\Gamma(M))^{\mathrm{Heegner}}$ extend to $\overline{\Sh}_\Gamma(M)$ and explicitly compute $\operatorname{Pic}(\overline{\Sh}_\Gamma(M))^{\mathrm{Heegner}}$. Building on Bruinier--Freitag's analysis of Heegner divisor extensions \cite{BF01}, we use theta series and their admissible isotropic lifts to  construct an obstruction subspace  within the space of cuspidal forms. Under some mild arithmetic conditions on $M$, we prove that this subspace is maximal.

Our main result is encapsulated in the following theorem:

\begin{theorem}[Cor. \ref{cor:mainthm2} and \ref{cor:mainthm3}]\label{thm:mainthm'}
    Let $M$ be an even lattice of signature $(2,n)$  with $n>8$ satisfying
    \begin{itemize}
        \item $M$ splits two hyperbolic planes 
        \item $M\otimes \ZZ_p$ splits three hyperbolic planes for all prime $p$. 
    \end{itemize} For any congruence subgroup $\Gamma$ of $\widetilde{\rO}(M) $,  we have  $$ \Pic_\QQ\left(\overline{\Sh}_\Gamma(M)\right)^{\rm Heegner}\cong \QQ,$$ 
  is spanned by some multiple of the extended Hodge line bundle $\overline{\lambda}$. Moreover, if $\Gamma=\widetilde{\rO}(M)$, then $$\Pic\left(\overline{\Sh}_\Gamma(M)\right)^{\rm Heegner}=\Pic\left(\overline{\Sh}_\Gamma(M)\right) \cong \ZZ.$$
\end{theorem}
Theorem \ref{thm:mainthm} is then a special case of Theorem \ref{thm:mainthm'} applied to the lattice \[
\Lambda_g \coloneqq \langle 2 - 2g \rangle \oplus E_8(-1)^{\oplus 2} \oplus U^{\oplus 2}.
\] 

\begin{remark}
When $\Gamma$ is neat, the Hodge bundle $\lambda$ extends as a Cartier divisor over $\overline{\Sh}_\Gamma(M)$ (cf.~\cite[Theorem 5.2.11]{Pera}).
\end{remark}

\begin{remark}[Conditions on $M$] Let us give some comments on the conditions given in Theorem \ref{thm:mainthm'}.   The local and global conditions imposed on  $M$ play distinct but crucial roles in our analysis:

 \begin{itemize}[leftmargin=*]
     \item (Global condition) The condition that $M$ splits two hyperbolic planes guarantees 
     \begin{enumerate}[label=(\alph*),nosep]
        \item the group  $\Pic(\Sh_\Gamma(M))^{\rm Heegner}$ admits a simple description in terms of modular forms;
        \item  when $\Gamma=\widetilde{\rO}(M)$, $\Pic(\Sh_\Gamma(M))=\Pic(\Sh_\Gamma(M))^{\rm Heegner}$ is generated by Heegner divisors. 
    \end{enumerate}

     \item (Local condition)  Consider $M = L \oplus U^{\oplus 2}$ with $L$ negative-definite. The requirement that $L \otimes \mathbb{Z}_p$ splits a hyperbolic plane guarantees surjectivity of certain theta liftings, which generally fails without this condition (see Remark \ref{rem:localConditionNecessary}).
     
      \item For $p$-elementary lattices $M$ which only split a hyperpbolic plane,  an analogous result holds when  $\Gamma=\mathrm{O}(M)$. The full statement and proof are provided in Theorem \ref{thm:mainthm2}.
 \end{itemize}

\end{remark}

\subsection{Moduli spaces as Baily--Borel compactifications}  
A central motivation for our study of the Picard group of Baily--Borel compactifications lies in the ubiquity of these compactifications in moduli theory. Many moduli spaces in algebraic geometry admit natural Baily--Borel compactifications or are themselves isomorphic to such compactifications. This is particularly true for moduli spaces of higher-dimensional varieties with special holonomy, such as hyper-K\"ahler varieties and some bounded Calabi--Yau surface pairs (cf.~\cite{Do96,La10,BL24}). 

Recent breakthroughs have established explicit identifications between Baily--Borel compactifications and moduli spaces of boundary polarized Calabi--Yau pairs. In \cite[Thm. 1.1]{BL24}, Blum--Liu construct a projective \textit{asymptotically good moduli space}
\[
\mathscr{M}^{\mathrm{CY}} = \left\{ (X,D)~~ \middle| 
\begin{array}{l}
\text{(1) } \dim X = 2, \, X \text{ projective slc} \\
\text{(2) } K_X + D \sim_{\mathbb{Q}} 0 \\
\text{(3) } D \text{ ample } \mathbb{Q}\text{-divisor} 
\end{array}
\right\}/\sim
\]  
parametrizing $S$-equivalence classes of boundary polarized Calabi--Yau surface pairs. This moduli space (or its seminormalization) serves as a wall-crossing space between K-stability and KSBA compactifications and carries a CM line bundle.

A fundamental example comes from K3 surfaces with non-symplectic involutions. For a K3 surface $S$ with non-symplectic involution $\sigma$ of type $\rho$, the quotient pair  
\[
\left(S/\sigma, \tfrac{1}{2}C\right)
\]  
where $C \subset S$ is the fixed locus of $\sigma$,  forms a boundary polarized CY pair.  Let $\scF_\rho$ be the coarse moduli space of such pairs, which has been extensively studied in the past few years (see \cite{AE22,AET23,AEH24,AEGS2025,BL24}). There are  various modular compactifications of $\scF_\rho$ which fit into a wall-crossing diagram:  

\begin{center}
\begin{tikzpicture}[node distance=1.8cm]
\node (K) {$\mathscr{F}^{\mathrm{K}}_\rho$};
\node (CY) [below of=K] {$\mathscr{F}^{\mathrm{CY}}_\rho$};
\node (KSBA) [left of=CY, xshift=-1.5cm] {$\mathscr{F}^{\mathrm{KSBA}}_\rho$};
\node (BB) [right of=CY, xshift=2.2cm] {$(\Gamma_\rho\backslash \bD_\rho)^{\mathrm{BB}}$};

\draw[->] (K) -- (CY) node[midway, right] {\footnotesize K-moduli};
\draw[->] (KSBA) -- (CY) node[midway, above] {\footnotesize KSBA};
\draw[->, line width=0.7pt] (BB) -- (CY) node[midway, above]  {\footnotesize  \cite[Thm. 1.3]{BL24}};
\end{tikzpicture}
\end{center}
Here $\mathscr{F}^{\mathrm{CY}}_\rho$ is the seminormalization of the closure of $\mathscr{F}_\rho$ in $\mathscr{M}^{\mathrm{CY}}$, $\bD_\rho$ is the type IV Hermitian symmetric domain associated to a lattice $\Lambda_\rho$, and $\Gamma_\rho=\rO(\Lambda_\rho)$ the orthogonal group. The birational morphism \[(\Gamma_\rho \backslash\bD_\rho)^{\rm BB}\to \scF_\rho^{\rm CY}\]
is the normalization map. As a by-product of Theorem \ref{thm:mainthm2}, we obtain the following. 

\begin{corollary}[Picard group of BCY moduli spaces]\label{cor:pic-cy}  
The rational Picard group of the normalization of \(\mathscr{F}^{\mathrm{CY}}_\rho\) is spanned by the CM line bundle provided that \(\dim \mathscr{F}_\rho > 8\) and $\mathscr{F}_\rho$ is not the 11-dimensional case, in which the fixed locus of \(\sigma\) is a smooth rational curve.
\end{corollary}  

The lattice $\Lambda_\rho$ is isomorphic to the coinvariant lattice of $\rH^2(S,\ZZ)$ under the antisymplectic involution, which is a $2$-elementary even lattice. We refer to \cite[3.6.2]{Ni79} and \cite[2A]{AE22} for a complete  classification of all $\Lambda_\rho$. 
Crucially, $\Lambda_\rho$ need not split two hyperbolic planes—a subtlety that necessitates our approach to Corollary \ref{cor:pic-cy}. This result follows from a dedicated analysis of obstruction spaces and the span of Heegner divisors, employing the study of admissible isotropic lifts of theta series.

\begin{remark}\label{rmk:nikulin-classification}
By \cite{Nikulin1983Factorgroups} and see also \cite{AE22}, there exist 75 families of K3 surfaces with non-symplectic involutions, of which 38 satisfy \(\dim \mathscr{F}_\rho > 8\). The exceptional case in Corollary \ref{cor:pic-cy} corresponds to the 11-dimensional family with lattice \(\Lambda_\rho \cong A_1 \oplus U(2) \oplus E_8(-2)\). Since our general framework does not extend to this case, we leave it open for further study, which may require alternative or more specialized methods.
\end{remark}

\subsection{The proof strategy} 
 Our main theorem (Theorem \ref{thm:mainthm}) amounts to characterizing precisely when Heegner divisors extend to the Baily--Borel boundary: this occurs \emph{if and only if} the divisor class is proportional to the Hodge bundle $\lambda$. The proof synthesizes modular forms and Shimura geometry through three key innovations:

\begin{enumerate}[leftmargin=*]

 \item (Obstructions theory) A \emph{local-to-global} criterion which combines Borcherds products \cite{Bo98} with Bruinier--Freitag's boundary analysis \cite{BF01} establishes the following principle between geometry and arithmetic: 

\begin{table}[h]
\renewcommand{\arraystretch}{1.2} 
\begin{tabular}{cclc}
\textbf{Geometry}                                                                                         & \multicolumn{2}{c}{}                                & \textbf{Arithmetic}                                                                                                                \\ \cline{1-1} \cline{4-4} 
\multicolumn{1}{|c|}{Heegner Divisors}                                                                    & \multicolumn{2}{c|}{$\longleftrightarrow$}          & \multicolumn{1}{c|}{\begin{tabular}[c]{@{}c@{}}Dual of\\ \quad Almost Cusp Forms \quad\quad\end{tabular}} \\ \cline{1-1} \cline{4-4} 
\multicolumn{1}{|c|}{\begin{tabular}[c]{@{}c@{}}Obstruction to Extending\\ Heegner divisors\end{tabular}} & \multicolumn{2}{c|}{$\longleftarrow\joinrel\rhook$} & \multicolumn{1}{c|}{\begin{tabular}[c]{@{}c@{}}(Admissible Isotropic Lifts of)\\ theta series\end{tabular}}                                   \\ \cline{1-1} \cline{4-4} 
\end{tabular}
\end{table}

    Roughly speaking, $\Pic_\QQ(\Sh_\Gamma(M))$ can be identified with the dual of almost cusp forms, while the obstruction space contains the space of theta series and their admissible isotropic lifts, which is a subspace of the space of cusp forms.

    \item (Basis problem for modular forms) We study whether the space of cusp forms is generated by theta series or their special isotropic lifts.  This is called the basis problem, which is a classical problem in the theory of modular forms and goes back to Eichler.  We resolve the basis problem in the vector-valued case for half-integral weight forms (Theorem \ref{thm-cusp=theta}), circumventing even signature restrictions in \cite{Ma24}.
    \item (Torsion-freeness) By extending the work in \cite{DFV24}, we establish the torsion-freeness of $\Pic(\Sh_\Gamma(M))$, also simplifying the earlier arguments in \cite{DFV24}.
   
\end{enumerate}

\subsection{Organization of the paper}  
The paper is organized as follows. Section \ref{sec2} establishes foundations of vector-valued Siegel modular forms, including theta series constructions and Siegel--Eisenstein series, developing essential tools for subsequent sections.  

Sections \ref{sec3} and \ref{sec4} form the technical core, where we prove surjectivity of theta liftings for odd-rank cases by adapting \cite{Ma24}'s methods through Bruinier--Stein and Shimura Hecke operators.  Some computational details are provided in the Appendix, and geometrically-oriented readers may proceed directly to Section \ref{sec5}.  

In Sections \ref{sec5} and \ref{sec6}, we establish the arithmetic obstruction subspace of extending Heegner divisors by 
\begin{itemize}
    \item [-] Extension of Eichler's criteria to some lattices splitting only one hyperbolic plane, this allows us to generalize some results in \cite{BLMM17};
    \item [-] Classifying boundary components of Baily--Borel compactifications and introducing so-called admissible isotropic planes, which will be used to construct the obstruction subspace; 
    \item [-] Using Bruinier--Freitag's local Bocherds product and admissible isotropic lifting to obtain an obstruction subspace for extending Heegner divisors in terms of various theta series.
\end{itemize}
The main theorem is proved in Section \ref{sec6}, synthesizing these technical components to resolve the Picard group computation, with applications to moduli spaces of Calabi--Yau pairs.

\subsection*{Notations and conventions}

\begin{itemize}[leftmargin=2em]
\setlength{\itemsep}{5pt}

\item  For any normal variety \( X \), denote by \(\Cl_\QQ(X)\) its \(\QQ\)-Weil divisor class group.

\item A discriminant form is a finite abelian group $G$ with a quadratic form $\frq: G \to \QQ/\ZZ$ such that \[
(\gamma, \delta) \coloneqq \frq(\gamma + \delta) - \frq(\gamma) - \frq(\delta)
\] is a non-degenerate symmetric bilinear form.  More generally, for any two elements $\bfu=(u_1,...,u_d), \bfv=(v_1,...,v_d)\in G^{(d)}=G\times G \times \cdots \times G$, we define $(\bfu,\bfv)\coloneqq\begin{pmatrix}
(u_i,v_j)_{i,j}
\end{pmatrix}  \in \sym_d(\QQ/\ZZ)$. 
The level of $G$ is \[\frn(G):=\min\{m\in\ZZ_{>0}\mid~ m\frq(x) = 0  \hbox{ for all } x\in G\}.\] The non-degeneracy of \((-,-)\) implies \(\frn(G) \cdot \gamma = 0\) for all \(\gamma \in G\). The \emph{\(p\)-rank} of \(G\) is \(\mathrm{rank}_p(G) \coloneqq \dim_{\mathbb{F}_p}(G \otimes \mathbb{F}_p)\).

\item If $G$ is 2-torsion,  one can follow \cite{AE22} to define  the \textit{coparity} of $G$ by $\delta_G=0$ if every
$2\frq(x)=0\mod \ZZ$ for all $x\in G$, otherwise $\delta_G=1$.  There is a unique element $\alpha\in G$ such that 
\[2\frq(\gamma)\equiv (\gamma,\alpha)\mod \ZZ \]
for any $\gamma\in G$, called the {\it characteristic element} of $G$.  Note that $G$ contains a nonzero, isotropic, characteristic element $\alpha$ if and only if $\delta_G=1$ and has signature $0\bmod4$. In this case, the coparity of $\alpha^\bot/\langle\alpha\rangle$ is $0$.

 \item Define:
\begin{align*}
G^n & \coloneqq \mathrm{image}(G \xrightarrow{\times n} G), \\
G_n & \coloneqq \ker(G \xrightarrow{\times n} G), \\
G^{n*} & \coloneqq \left\{ \gamma \in G \mid n\frq(\mu) + (\mu, \gamma) \equiv 0 \pmod{\mathbb{Z}},\, \forall \mu \in G_n \right\}.
\end{align*}

\item Let $\CC[G^{(d)}]$ be the group algebra of $G^{(d)}$.  There is a standard Hermitian inner product $\langle-,-\rangle$ on $\CC[G^{(d)}]$ given by 
\begin{equation}
        \langle \sum\limits_{\gamma\in G^{(d)}} a_\gamma \fre_\gamma, \sum\limits_{\gamma\in G^{(d)}} b_\gamma \fre_\gamma\rangle=\sum\limits_{\gamma\in G^{(d)}} a_\gamma\overline{b}_\gamma.
\end{equation}
By identifying $\CC[G^{(2)}]$ with $\CC[G]\otimes \CC[G]$, we also use $\langle -,-\rangle$ to denote the antilinear map $\CC[G]\times \CC[G^{(2)}]\to \CC[G]$ given by 
\begin{equation}
    \langle v,u\otimes w\rangle =\langle v,u \rangle \cdot \bar{w}
\end{equation}
where $u,v,w\in\CC[G]$.

    \item  Let \((M, \langle \cdot, \cdot \rangle)\) be an even lattice. We can identify the dual lattice $M^\vee$ as the sublattice $$\{x\in M\otimes\QQ\mid~ \left<x,y\right>\in\ZZ~ \hbox{for all } y\in M\}$$ of $M \otimes \QQ$.
The \emph{discriminant group} \(G_M \coloneqq M^\vee/M\) satisfies \(|G_M| = |\det M|\). The quadratic form \(\frq_M: G_M \to \mathbb{Q}/\mathbb{Z}\) is given by
\[
\frq_M(v) \equiv \tfrac{1}{2} \langle v, v \rangle \mod \mathbb{Z}, \quad v \in M^\vee.
\]
Let \(\ell(G_M)\) denote the minimal number of generators of \(G_M\).

\item For a discriminant form \(G\), let \(\mathbf{Gen}_{b^+,b^-}(G)\) denote the genus of even lattices with signature \((b^+,b^-)\) and discriminant form \(G\). The \emph{signature} of \(G\) is defined as
\[
\mathrm{sign}(G) \coloneqq b^+ - b^- \pmod{8} \in \mathbb{Z}/8\mathbb{Z},
\]
for \((b^+,b^-)\) such that the genus \(\mathbf{Gen}_{b^+,b^-}(G)\) is non-empty and the signature is constant over all genera with discriminant form \(G\).

\item Define \(\bfe(x) \coloneqq e^{2\pi i x}\) and \(\sqrt{z} \coloneqq z^{1/2}\) (principal branch on \(\mathbb{C}\)).
\end{itemize}

\section{Vector-valued Siegel modular forms}\label{sec2}

\subsection{Weil Representation}\label{subsec:weil-rep}
Let $\mathbb{H}_d$ denote the Siegel upper half-space of degree $d$. The \emph{metaplectic double cover} $\Mp_{2d}(\ZZ)$ of $\Sp_{2d}(\ZZ)$ consists of elements $\widetilde{g} = \bigl(g, \phi_g(\tau)\bigr)$, where for $g = \begin{pmatrix} A & B \\ C & D \end{pmatrix} \in \Sp_{2d}(\ZZ)$, the function $\phi_g(\tau)$ is holomorphic on $\mathbb{H}_d$ and satisfies $\phi_g(\tau)^2 = \det(C\tau + D)$. The group $\Mp_{2d}(\ZZ)$ is generated by
\[
J_d = \Bigl( \MJ, \sqrt{\det \tau} \Bigr) 
\quad\text{and}\quad 
n(B) = \Bigl( \MU, 1 \Bigr),
\]
where $B \in \sym_d(\ZZ)$.

Let $G$ be a discriminant form. For each $d \in \ZZ_{\geq 1}$, the \emph{Weil representation} $\rho_{G}^{(d)}$ is a unitary representation of $\Mp_{2d}(\ZZ)$ on the group ring $\CC[G^{(d)}]$. It is defined by
\begin{equation}
\begin{aligned}
\rho_{G}^{(d)}(n(B))\fre_{\gamma} 
    &= \bfe\left( \tfrac{1}{2} \tr\left( (\gamma,\gamma) B \right) \right) \fre_{\gamma}, \\
\rho_{G}^{(d)}(J_d)\fre_{\gamma} 
    &= \frac{ \bfe\left( -\frac{d}{8} \sign(G) \right) }{ |G|^{d/2} } 
       \sum_{\delta \in G^{(d)}} \bfe\bigl( -\tr(\gamma,\delta) \bigr) \fre_{\delta},
\end{aligned}
\end{equation}
where $\fre_\gamma$ denotes the standard basis vector in $\CC[G^{(d)}]$ for $\gamma = (\gamma_1,\dots,\gamma_d) \in G^{(d)}$.

\begin{remark}
For $m(U) = \Bigl( \MP, \sqrt{\det U^{-1}} \Bigr) \in \Mp_{2d}(\ZZ)$, we have
\[
\rho_{G}^{(d)}(m(U))\fre_{\gamma} 
    = \sqrt{\det U^{-1}}^{\sign(G)} \fre_{\gamma U^{-1}}.
\]
\end{remark}

\paragraph{\bf Conventions:} When $G \cong M^\vee/M$ for a lattice $M$, we write $\rho_{M}^{(d)}$ for $\rho_{G}^{(d)}$. For $d=1$, we simplify this to $\rho_M$.

\subsection{Siegel Modular Forms}

For $k \in \frac{1}{2}\ZZ$, define the Petersson slash operator $\mid_{k,G^{(d)}}$ on vector-valued functions $f \colon \mathbb{H}_d \to \CC[G^{(d)}]$ by
\[
(f \mid_{k,G^{(d)}} [\widetilde{g}])(\tau) \coloneqq 
    \phi_g(\tau)^{-2k} \det(g)^{k/2}  \rho_{G}^{(d)}(\widetilde{g})^{-1} f(g\tau),
\]
where $\widetilde{g} = (g, \phi_g(\tau)) \in \widetilde{\GL_2^+}(\RR)$ with
\[\widetilde{\GL_2^+}(\RR) \coloneqq 
    \Biggl\{ \Bigl( g = \begin{pmatrix} a & b \\ c & d \end{pmatrix}, \phi_g(\tau) \Bigr) \in \GL_2^+(\RR) \times \mathcal{O}(\mathbb{H}_d)
    \Bigm| \exists t \in \CC^\times, \, \phi_g(\tau)^2 = t (c\tau + d) \Biggr\}\]
denoting the metaplectic $\CC^\times$-extension of $\GL_2^+(\RR)$\footnote{This convention differs from Shimura's \cite{Shimura73}, following Bruinier--Stein \cite{BS}.}.

\begin{definition}
    A \emph{Siegel modular form} of weight $k$ and type $\rho_{G}^{(d)}$ is a holomorphic function $f \colon \mathbb{H}_d \to \CC[G^{(d)}]$ satisfying:
    \begin{enumerate}[label=(\roman*), leftmargin=*, nosep]
        \item $f \mid_{k,G^{(d)}} [\widetilde{g}] = f$ for all $\widetilde{g} \in \Mp_{2d}(\ZZ)$;
        \item $f$ is holomorphic at cusps.
    \end{enumerate}
\end{definition}

Such forms admit Fourier expansions
\[
f(\tau) = \sum_{\gamma \in G^{(d)}} 
    \sum_{\substack{T \in \sym_d(\ZZ)_{\mathrm{even}} + (\gamma,\gamma) \\ T \geq 0}} 
        c(T, \gamma) \bfe\bigl( \tfrac{1}{2} \tr(T\tau) \bigr) \fre_\gamma.
\]
We call $f$ a \emph{cusp form} when $c(T,\gamma) \neq 0 \Rightarrow T > 0$.

Denote by $\Mod_k(\rho_{G}^{(d)})$ and $\Cusp_k(\rho_{G}^{(d)})$ the spaces of Siegel modular forms and cusp forms, respectively. The central element $(\mathrm{I}_{2d}, -1)$ acts via $(-1)^{\sign(G)}$, implying
\[
\Mod_k(\rho_{G}^{(d)}) = 0 \quad \text{if} \quad 2k + \sign(G) \not\equiv 0 \pmod{2}.
\]
We therefore universally assume the parity condition
\begin{equation}\label{parity}
    2k + \sign(G) \equiv 0 \pmod{2}.
\end{equation}

For $d=1$, the Petersson inner product on $\Mod_k(\rho_M)$ is
\[
\langle f, g \rangle_{\mathrm{Pet}} \coloneqq 
    \int_{\SL_2(\ZZ) \backslash \mathbb{H}} \langle f(z), g(z) \rangle y^{k-2}  \mathrm{d}x\mathrm{d}y,
\]
well-defined when at least one of them is cuspidal.

\subsection{Some functoriality results regarding Weil representations}

This subsection gathers additional notions and results related to various Weil representations for future reference. 
\subsection*{Isotropic lift/descent}

Let $H\leq G$ be an isotropic subgroup.
\begin{itemize}[leftmargin=2em]
    \item The isotropic lift map $\uparrow_H:\CC[H^\perp/H]\to \CC[G] $ is defined by 
    \[\uparrow_H(\fre_{\gamma+H})=\sum\limits_{\mu\in H} \fre_{\gamma+\mu}\] 
for all $\gamma\in H^\perp$.
\item The isotropic descent map $\downarrow_H:\CC[G]\to \CC[H^\perp/H]$ is defined by 
\[
\downarrow_H(\fre_\gamma) = 
\begin{cases} 
\fre_{\gamma+H} & \text{if } \gamma \in H^\perp, \\
0 & \text{otherwise}.
\end{cases}
\]
\end{itemize}   

It is well known that they are adjoint with respect to the standard inner products on $\CC[G]$ and $\CC[H^\perp/H]$ respectively. Further, they commute with Weil representations, which implies that they take modular forms to modular forms and preserve cuspidality.

\subsection*{Embeddings and tensor products}

For $0<r< d$, there is an inclusion map
\begin{equation}
  \iota:  \Mp_{2r}(\ZZ)\times \Mp_{2d-2r}(\ZZ)\to \Mp_{2d}(\ZZ) 
\end{equation}
given by
\[
 \Big(\left( \begin{matrix} A & B \\ C & D \end{matrix} \right), \phi \Big), 
\Big(\left( \begin{matrix} A' & B' \\ C' & D' \end{matrix} \right), \phi' \Big)  
\longmapsto 
\Bigg(\left( \begin{matrix} A &  & B &  \\  & A' &  & B' \\
C &  & D & \\
& C' & & D' 
\end{matrix} \right), 
\widetilde{\phi} \Bigg),
\] where $\widetilde{\phi}(\diag(z,z')) 
= \phi(z) \cdot \phi'(z')$. This embedding satisfies the following compatibility with the Weil representation 
\begin{equation*}
    \rho_G^{(d)}(\iota(\tg_1,\tg_2))(\fre_{\gamma}\otimes \fre_{\gamma'})=\rho_G^{(r)}(\tg_1)(\fre_\gamma)\otimes \rho_G^{(d-r)}(\tg_2)(\fre_{\gamma'}),
\end{equation*}
where we identify $\CC[G^{(r)}]\otimes \CC[G^{(d-r)}]$ with $\CC[G^{(d)}]$.(cf.~\cite[Lem. 2.2]{BZ24})

\subsection{Siegel--Eisenstein series and theta series}
Suppose $M$ is a positive definite even lattice of rank $r$ with bilinear form $(-,-)$. Fix an embedding of $M$ into $\RR^r$ such that $(-,-)$ coincides with the standard inner product on $\RR^r$.
\subsection*{Siegel--Eisenstein series}
Let $\Gamma^{(d)}_\infty\subseteq \Sp_{2d}(\ZZ)$ be the subgroup generated by 
\[
    \MU \text{ and } \MP 
\]
for $B\in \sym_d(\ZZ)$ and $U\in \GL_d(\ZZ)$.  Let $\widetilde{\Gamma}_\infty^{(d)}\subseteq \Mp_{2d}(\ZZ)$ be the preimage of $\Gamma_\infty^{(d)}$ under the metaplectic cover. For $k\in\frac{1}{2}\ZZ$ and $k>d+1$, the vector-valued Siegel--Eisenstein series is defined by 
$$\begin{aligned}
    \bE^{(d)}_{k,M}(\tau)&=\sum\limits_{\tilde{g}= (g,\phi_g(\tau))\in \widetilde{\Gamma}_\infty^{(d)}  \backslash \Mp_{2d}(\ZZ)} \fre_0\mid_{k,G^{(d)}_M}[\tg] \\
&=\sum\limits_{\tilde{g}= (g,\phi_g(\tau))\in \widetilde{\Gamma}_\infty^{(d)}  \backslash \Mp_{2d}(\ZZ)}\phi_g(\tau)^{-2k}\cdot \rho_{M}^{(d)}(\tilde{g})^{-1}(\fre_0),
\end{aligned}$$
which is well-defined and converges normally. It is a vector-valued modular form of weight $k$ and type $\rho_{M}^{(d)}$. The summand $\fre_0\mid_{k,G_M^{(d)}}[\tg]$ is $\widetilde{\Gamma}_\infty^{(d)}$-invariant since the parity condition \eqref{parity} holds.  

\subsection*{Theta Series}

\begin{definition}
Let \(F(X)\) be a harmonic polynomial in the matrix variable \(X = (X_{ij})_{\substack{i=1,\dots,r \\ j=1,\dots,d}}\) of degree \(h\). We define the theta series
\begin{equation}\label{eq:theta}
    \Theta_{M,F}^{(d)}(\tau) = \sum_{\mathbf{v} \in (M^\vee)^{(d)}} F(\mathbf{v}) \mathbf{e}\left( \frac{1}{2} \tr\left( (\mathbf{v}, \mathbf{v}) \tau \right) \right) \fre_{\mathbf{v} + M^d}
\end{equation}
and the genus theta series
\begin{equation}\label{eq:gen-theta}
    \Theta_{\mathbf{Gen}(M)}^{(d)} = \frac{ 
        \sum\limits_{L \in \mathbf{Gen}(M)} |\mathrm{Aut}(L)|^{-1} \sum\limits_{\sigma \in \mathrm{Iso}(G_M, G_L)} \sigma^{*} \Theta^{(d)}_{L,1} 
    }{ 
        |\mathrm{Aut}(G_M)| \sum\limits_{L \in \mathbf{Gen}(M)} |\mathrm{Aut}(L)|^{-1} 
    },
\end{equation}
where \(\mathbf{Gen}(M)\) denotes \(\mathbf{Gen}_{r,0}(G_M)\), and \(\sigma^{*} \fre_{\gamma} = \fre_{\sigma^{-1}(\gamma)}\) for \(\gamma \in G_L^{(d)}\).

\end{definition}

By \cite[Thm. 4.1]{Bo98}, \(\Theta^{(d)}_{M,F} \in \mathrm{Mod}_{h+\frac{r}{2}}(\rho_{M}^{(d)})\). Moreover, \(\Theta^{(d)}_{M,F}\) is a cusp form when \(h > 0\). Define the subspaces
\[
\Mod^\theta_k(\rho_{M}^{(d)}) \coloneqq \mathrm{Span} \left\{ \sigma^{*}\Theta_{L,F}^{(d)} \;\middle|\; 
\begin{array}{c} 
    L \in \mathbf{Gen}(M), \\ 
    F \text{ harmonic of degree } k - \frac{r}{2}, \\ 
    \sigma \in \mathrm{Iso}(G_M, G_L)
\end{array} \right\} \subseteq \Mod_k(\rho_{M}^{(d)}),
\]
and \(\Cusp^\theta_k(\rho_{M}^{(d)}) \coloneqq \Mod^\theta_k(\rho_{M}^{(d)}) \cap \Cusp_k(\rho_{M}^{(d)})\). The central question addressed in the next two sections is whether 
\[
\Cusp^\theta_k(\rho_{M}^{(d)}) = \Cusp_k(\rho_{M}^{(d)}),
\]
i.e., whether every cusp form is a linear combination of theta series. This was established for even-rank lattices in \cite{Ma24}. We adapt the proof to the odd-rank case.

For a negative definite lattice \(M\), define
\[
\Theta_{M,F}^{(d)} \coloneqq \Theta_{M(-1),F}^{(d)} \in \mathrm{Mod}_{h+\frac{r}{2}}(\rho_{M(-1)}^{(d)}) = \mathrm{Mod}_{h+\frac{r}{2}}(\rho_{M}^{(d)*}),
\]
where \(\rho_{M}^{(d)*}\) denotes the dual representation of \(\rho_{M}^{(d)}\).

\begin{remark}\label{rmk:quadric-harmonic}
When \(d=1\), the space of harmonic polynomials in \(r\) variables of degree \(h\) has dimension 
\[
\binom{r+h-1}{r-1} - \binom{r+h-3}{r-1}.
\]
An explicit basis can be constructed via the Kelvin transform (cf. \cite[Thm. 5.25]{ABR01}). For \(h=2\), harmonic polynomials are linear combinations of 
\begin{equation}  
F_u(v) = \langle u, v \rangle^2 - \frac{\langle u, u \rangle \langle v, v \rangle}{r}
\end{equation}
for \(u \in M\).
\end{remark}

The genus theta series can be identified with the Siegel--Eisenstein series via the Siegel--Weil formula.

\begin{theorem}\label{thm:Eis=Theta}
Let \(M\) be a positive definite even lattice of rank \(r\) with \(\frac{r}{2} > d + 1\). Then
\[
\Theta^{(d)}_{\mathbf{Gen}(M)} = \mathbf{E}_{\frac{r}{2},M}^{(d)}.
\]
\end{theorem}
\begin{proof}
The proof appears in Appendix \ref{appA} (Theorem \ref{thm:vecSW}).
\end{proof}

\section{Vector-valued Hecke operator}\label{sec3}

\subsection{Bruinier--Stein's vector-valued Hecke operator}

Let $G$ be a finite abelian group equipped with a discriminant form of level $\frn(G)=:N$. Bruinier and Stein have introduced a Hecke operator acting on $\Mod_k(\rho_G)$. For each positive integer $\alpha>0$, set $$ \btg_\alpha=\Big(\left(\begin{array}{cc}\alpha^2 & 0 \\0 & 1\end{array}\right) ,1\Big)\in \widetilde{\GL_2^+}(\RR).$$ 
The (inverse of the) Weil representation can be extended to the double coset \[\widetilde{\bY}_{\alpha^2}\coloneqq \Mp_2(\ZZ) \cdot \btg_\alpha  \cdot \Mp_2(\ZZ).\] For any element  $\tg\cdot  \btg_\alpha \cdot \tg'\in \widetilde{\bY}_{\alpha^2}$, the action is defined by:
\[\fre_\lambda\mid [\tg\cdot  \btg_\alpha \cdot \tg']\coloneqq \fre_\lambda\mid _G[\tg]\cdot\mid_G[\btg_\alpha]\cdot\mid_G[\tg'],\] 
where $\fre_\gamma \mid _G[\tg]=\rho_G(\tg)^{-1}(\fre_\gamma)$ is the Weil representation  and $\fre_\gamma \mid_G[\btg_\alpha]=\fre_{\alpha\gamma}$. This action is well-defined by \cite[Proposition 5.1]{BS}.

\begin{remark}
    It is easy to see that this extension can also be made in the case when $\alpha=0$. This further extension will be used in Section \ref{sec:mod-theta}.
\end{remark}

Armed with this extended representation, we can define the vector-valued Hecke operator as follows.

\begin{definition}[{\cite[Def. 5.5]{BS}}]\label{hecke}
Let $\alpha>0$ be a positive integer. Express $\widetilde{\bY}_{\alpha^2}$ as a disjoint union of left cosets 
   $$\bigsqcup\limits_{i} \Mp_2(\ZZ)\cdot\tilde{\delta}_i$$
for  some $\tilde{\delta}_i=(\delta_i,\phi_{\delta_i}(\tau))\in \widetilde{\bY}_{\alpha^2}$. The Hecke operator 
  $ \bT_{\alpha^2}: \Mod_k(\rho_G)\to \Mod_k(\rho_G)$
is defined by $$\bT_{\alpha^2}(f)\coloneqq \alpha^{k-2}\sum_i \sum_{\gamma\in G}(f_\gamma \mid_k[\tilde{\delta}_i])\otimes (\fre_\gamma\mid [\tilde{\delta}_i]), $$ where $f=\sum\limits_{\gamma\in G} f_\gamma\otimes \fre_{\gamma}$. 

\end{definition}

This definition is proved to be well-defined in \cite[Section 5]{BS}. The main properties of these Hecke operators are summarized in the following proposition.

\begin{proposition}[{\cite[Thm. 4.12, Thm 5.6]{BS}}]\label{prop:hecke-alg}
    The Hecke operators defined in Definition \ref{hecke} satisfy the following properties:
    \begin{itemize}[leftmargin=2em]
     \item These Hecke operators take cusp forms to cusp forms. 
     \item They are self-adjoint with respect to the Petersson inner product, and the Hecke operators $$\{\bT_{\alpha^2}:\gcd(\alpha, N)=1\}$$ generate a commutative subalgebra of $\operatorname{End}_{\CC}(\Cusp_k(\rho_G))$.
     \item For each pair $\alpha,\beta$ of coprime positive integers, $\bT_{\alpha^2}\circ\bT_{\beta^2}=\bT_{(\alpha\beta)^2}$.
    \end{itemize}
\end{proposition}

As a consequence, one has 
\begin{corollary}\label{cor:basis}
There exists a basis of $\Cusp_k(\rho_G)$ consisting of simultaneous eigenforms for all Hecke operators in $\{\bT_{\alpha^2}:\gcd(\alpha, N)=1\}$.
\end{corollary}

\subsection{Non-vanishing of $L$-values}
Let $f\in\Cusp_k(\rho_G)$ be a non-zero simultaneous eigenform for all Hecke operators $\{\bT_{\alpha^2}:\gcd(\alpha, N)=1\}$ with eigenvalues $\{\lambda(\alpha^2):\gcd(\alpha, N)=1\}$. One can define the $L$-series
\begin{equation}\label{L-series}
    L(f,s)\coloneqq \sum_{\substack{\alpha\geq 1 \\ \gcd(\alpha, N)=1}}\frac{\lambda(\alpha^2)}{\alpha^s}.
\end{equation}
The main analytic properties of the $L$-series relevant to our goals are summarized in the following theorem.

\begin{theorem}\label{thm:nonvanish}
   Suppose $k\geq2$. The $L$-series \eqref{L-series} converges absolutely for $\mathrm{Re}(s)>k+1$. In particular, when $\mathrm{Re}(s)>k+1$,
   \[L(f,s)\neq 0.\]
\end{theorem}

\begin{proof}
The case when $k\in\ZZ$ is \cite[Proposition 3.3]{Ma24} and the result is slightly stronger. Now we consider the case when $k\in\ZZ+\frac{1}{2}$, then $\sign(G)$ is odd.  We will establish the theorem by comparing Bruinier--Stein's Hecke operators with the scalar-valued Hecke operators defined by Shimura \cite{Shimura73}.  

Since $\sign(G)$ is odd, the oddity formula \cite[p. 383 (30)]{CS87} implies that $4\mid N$. There is a map \[\Gamma_0(4)\stackrel{\kappa}{\longrightarrow} \widetilde{\GL^+_2}(\RR)\] given by
        \[\begin{pmatrix}
            a&b\\c&d
        \end{pmatrix}\mapsto \Big(\begin{pmatrix}
            a&b\\c&d
        \end{pmatrix}, \epsilon_d^{-1}\left(\frac{c}{d}\right)\sqrt{c\tau+d}\Big),\]
        where $\epsilon_d=1$ if $d\equiv 1\bmod 4$, $\epsilon_d=i$ if $d\equiv -1\bmod 4$, and $\left(\frac{\cdot}{\cdot}\right)$ is the Jacobi symbol. This map lifts to $\Mp_2(\ZZ)$ when restricted to $\Gamma_1(4)$.
       We denote by  $\Delta(N)$ the image of any principal congruence subgroup $\Gamma(N)\subset \Gamma_1(4)$ under this lift. For a character $\chi$ of $(\ZZ/N\ZZ)^\times$ and $\gamma\in G$, define the element $$v_{\gamma,\chi}\coloneqq \sum\limits_{d\in(\ZZ/N\ZZ)^\times}\chi(d)^{-1}\fre_{d\gamma}.$$ Since the Weil representation is trivial on $\Delta(N)$, the function $F_{\gamma,\chi}\coloneqq \langle f,v_{\gamma,\chi} \rangle$ is then an element in $\Cusp_k(\Delta(N))$, where  $\Cusp_k(\Delta(N))$ is the space of scalar-valued cusp forms of weight $k$ and level $\Delta(N)$  (cf.~\cite{Shimura73}). 

We extend $\kappa$ to 
\[ \bigcup_{\alpha=1}^\infty\Gamma_0(4)\begin{pmatrix}
    \alpha^2&0\\0&1
\end{pmatrix}\Gamma_0(4)\stackrel{\kappa}{\longrightarrow} \widetilde{\GL^+_2}(\RR) \]
by $\left(\begin{smallmatrix}
    \alpha^2&0\\0&1
\end{smallmatrix}\right)\mapsto\btg_\alpha$. For every prime $p$ satisfying $\gcd(p,N)=1$, choose an element $R_{p^{2n-a}}\in \SL_2(\ZZ)$ as a lift of the element $\left(\begin{smallmatrix}
        p^{-2n+a}&0\\0&p^{2n-a}
    \end{smallmatrix}\right)\in\SL_2(\ZZ/N\ZZ)$ and consider the set
 \[
        \Sigma_{p^{2n}}\coloneqq \left\{\kappa\left(R_{p^{2n-a}}\begin{pmatrix}
        p^{2n-a}&bN\\0&p^{a}
    \end{pmatrix}\right)\mid 0\leq a\leq 2n\right., 
    \left.0\leq b<p^{a},~\gcd(b,p^{\min\{a,2n-a\}})=1\right\}.\]
    Since the matrix $R_{p^{2n-a}}\left(\begin{smallmatrix}
        p^{2n-a}&bN\\0&p^{a}
    \end{smallmatrix}\right)$ is congruent to $\left(\begin{smallmatrix}
        1&0\\0&p^{2n}
    \end{smallmatrix}\right)$ modulo $N$, its image under $\kappa$ lies in the metaplectic \textit{double} cover of $\GL_2(\RR)$. As shown in \cite[\S 4]{Wr14},  this set can serve as a system of representatives for
\[
        \Mp_2(\ZZ)\backslash\Mp_2(\ZZ)\left(\begin{pmatrix}p^{2n} & 0 \\0 & 1\end{pmatrix} ,1\right)\Mp_2(\ZZ)\cong \Delta(N)\backslash\Delta(N)\left(\begin{pmatrix}1 & 0 \\0 & p^{2n}\end{pmatrix} ,p^n\right)\Delta(N).\]
      
The advantage of choosing $\Sigma_{p^{2n}}$ as coset representatives is that the actions of these representatives via the Weil representation are particularly simple. Let $\tilde{g}=\left(g=\left(\begin{smallmatrix}
        a&b\\c&d
    \end{smallmatrix}\right),\phi_g(\tau)\right)\in \Mp_2(\ZZ)$ with $N\mid b,~N\mid c$ . For $\gamma\in G$, we have $\rho_G(\tilde{g})\fre_{\gamma}=\chi_G(\tilde{g})\fre_{d\gamma}$ for an explicit character $\chi_G$ (cf.~\cite[Thm. 5.4]{Bo00}). It follows that for every $\tilde{g}\in\Sigma_{p^{2n}}$, we have
    \[\fre_{\gamma}\mid [\tilde{g}]=
    \fre_{p^{-n}\gamma}.\] 

In \cite{Shimura73}, Shimura introduced a Hecke operator 
    \begin{equation*}
        \cT^{\Delta(N)}_{\alpha^2}:\Cusp_k(\Delta(N))\to \Cusp_k(\Delta(N))
    \end{equation*}
   defined by 
   \[\cT^{\Delta(N)}_{\alpha^2}(F)=\alpha^{k-2}\sum\limits_{\tg\in \Delta(N)\backslash\Delta(N)\left(\begin{pmatrix}\begin{smallmatrix}1 & 0 \\0 & \alpha^{2n}\end{smallmatrix}\end{pmatrix},\alpha^n\right)\Delta(N)} F\mid_{k} [\tg]. \]
   By the discussions above, it is easy to see that 
   \[\langle \bT_{\alpha^2}(f), v_{\gamma,\chi}\rangle =\chi(p^{n})^{-1}
   \cT^{\Delta(N)}_{p^{2n}}(F_{\gamma,\chi}). \]
    Now suppose $f$ is a non-zero simultaneous eigenform for all Hecke operators in $\{\bT_{\alpha^2}:\gcd(\alpha, N)=1\}$ with eigenvalues $\lambda(\alpha^2)$. Since the elements $v_{\gamma,\chi}$ span $\CC[G]$ and $f\neq 0$, one can find a pair $(\gamma,\chi)$ such that $F_{\gamma,\chi}\neq 0$. Then $F_{\gamma,\chi}$ is a non-zero simultaneous eigenform for all Hecke operators in $\{\cT^{\Delta(N)}_{\alpha^2}:\gcd(\alpha, N)=1\}$ with eigenvalues $\chi(\alpha)
    \lambda(\alpha^2)$. We claim that the series
    $$\sum_{\substack{\alpha\geq 1 \\ \gcd(\alpha, N)=1}}\frac{\chi(\alpha)
    \lambda(\alpha^2)}{\alpha^s}$$
    is absolutely convergent for $\mathrm{Re}(s)>k+1$. Hence $L(f,s)$ is also absolutely convergent for $\mathrm{Re}(s)>k+1$. Since we have a formal product expansion
    \[L(f,s)=\prod_{\gcd(p, N)=1}\sum_{n=0}^{\infty}{\frac{\lambda(p^{2n})}{p^{ns}}},\]
    general theory regarding Euler products now shows that the absolute convergence for $\mathrm{Re}(s)>k+1$ implies $L(f,s)\neq 0$.

    Now we prove the claim. First observe that $F_{\gamma,\chi}(N\cdot)\in\Cusp_k(\Delta_0(N^2),\psi)$ for some character $\psi$ of $(\ZZ/N^2\ZZ)^\times$. Furthermore, we have
    \[ (\cT_{\alpha^2}^{\Delta(N)}F_{\gamma,\chi})(N\cdot)=\cT_{\alpha^2}^{\Delta_0(N^2),\psi}(F_{\gamma,\chi}(N\cdot)). \]
    This is an analog of the classical way of treating a modular form of level $\Gamma(N)$ as a modular form of level $\Gamma_0(N^2)$ with a Nebentypus. Now let $p$ be a prime such that $\gcd(p,N)=1$. By the theory of Shimura lifts, $\chi(p)
    \lambda(p^2)$ is the $p$-th Fourier coefficient of a normalized cuspidal Hecke eigenform of weight $2k-1$. Hence the Ramanujan conjecture proved by Deligne implies that
    \[\left|\chi(p)
    \lambda(p^2)\right|\leq 2p^{k-1}.\]Since $p\geq 3$, by the explicit algebraic relations between the Hecke operators $\{\cT_{p^{2n}}^{\Delta_0(N^2),\psi}\mid n\geq 0\}$ in \cite[Thm. 1.1]{Pu14}, one can inductively prove that \[\left|\chi(p^{n})
    \lambda(p^{2n})\right|\leq 3^np^{n(k-1)}\leq p^{nk}.\] This immediately yields the claim.

\end{proof}

\subsection{The case when $p\mid N$}
The Hecke operators $\{\bT_{\alpha^2}:\gcd(\alpha, N)\neq 1\}$ are trickier since the Hecke algebra generated by $\{\bT_{p^{2n}}:n\geq 1\}$ is in general non-commutative if $p\mid N$ (cf. \cite[Remark 5.7]{BS}). However, regarding our goals, a comparison result of these Hecke operators with certain scalar-valued Hecke operators is sufficient. 

\begin{lemma}\label{lem:vec-sca}
    Fix a prime $p$. Let $f\in\Mod_k(\rho_G)$ and let $\gamma\in G$ satisfy the following conditions $(\ast_p)$ for $p$:
    \begin{itemize}[leftmargin=2em]
        \item $\gamma$ is not divisible by $p$ in $G$;
        \item when $p=2$, there exists $\mu\in G$, $2\mu=0$, such that  $2 \frq(\mu)+(\mu,\gamma)\neq 0\bmod \ZZ$.
    \end{itemize}
        Then 
        \[\langle\bT_{p^{2n}}(f),\fre_\gamma\rangle=\rT^{\frq(\gamma)}_{p^{2n}}\langle f,\fre_{p^{2n}\gamma}\rangle,\]
        where 
        \[\rT^{\frq(\gamma)}_{p^{2n}}(F)\coloneqq p^{(k-2)n}\sum_{b=0}^{p^{2n}-1}\bfe(-b\frq(\gamma))F\mid_k\left[\Big(\left(\begin{array}{cc}1 & b \\0 & p^{2n}\end{array}\right),p^n\Big)\right]\]
        for any scalar-valued function $F$ on $\HH$.
        
\end{lemma}

\begin{proof}
When $\alpha=p^n$,  there is an explicit coset decomposition $$ \widetilde{\bY}_{p^{2n}}=\bigcup_{a,b}\Mp_2(\ZZ) \cdot \Big(\begin{pmatrix}
        p^{2n-a} & b\\ 0 &  p^a
    \end{pmatrix}, \sqrt{p^a}\Big),$$
where $0\leq a\leq 2n,~0\leq b<p^{a}$ and $~\gcd(b,p^{\min\{a,2n-a\}})=1$ (cf.~\cite[\S 5.1]{St15}). Set $\tilde{\delta}_{a,b}=\left(\left(\begin{smallmatrix}
        p^{2n-a} & b\\ 0 &  p^a
    \end{smallmatrix}\right), \sqrt{p^a}\right)$. Then we have 
$$\frac{1}{p^{(k-2)n}}\langle \bT_{p^{2n}}(f), \fre_\gamma \rangle= \sum\limits_{a=0}^{2n}\!\sum\limits_{\substack{0\leq b<p^{a}\\ \gcd(b,p^{\min\{a,2n-a\}})=1}} \! \sum\limits_{\mu\in G} \langle 
 f_\mu\mid_k[\tilde{\delta}_{a,b}]\otimes (\fre_\mu\mid [\tilde{\delta}_{a,b}]), \fre_\gamma\rangle. $$
From the definition of $\rT_{p^{2n}}^{\frq(\gamma)}$, it suffices to show $$\langle \fre_\mu\mid [\widetilde{\delta}_{a,b}] , \fre_\gamma\rangle=0 $$ for every $a<2n$. The action of  $\tilde{\delta}_{a,b}$ is well-understood. For instance,  if $p$ is odd, there is an explicit formula of $\fre_\mu\mid [\tilde{\delta}_{a,b}]$ proved in \cite[Thm. 5.2]{St15} and \cite[Proposition 5.3]{BCJ19}. In general, it is of the form 
\begin{equation}\label{eq:weil-general}
    \fre_\mu\mid [\tilde{\delta}_{a,b}]=\sum_{\lambda} c_\lambda\fre_\lambda
\end{equation}
with $\lambda \in G^p$ if $p$ is odd, and $\lambda\in G^2\cup G^{2\ast}$ 
if $p=2$ (cf.~\cite[Thm. 4.7]{Sc09} and \cite[Thm. 1]{Str13}).  As $\gamma$ satisfies the condition $(\ast_p)$, it follows that $\langle \fre_\mu\mid [\widetilde{\delta}_{a,b}] , \fre_\gamma\rangle=0 $ for every $a<2n$. 
\end{proof}

By using this lemma inductively, one formally obtains a useful result for later detecting when a modular form is annihilated by those Hecke operators not coprime to $N$.
\begin{proposition}[{\cite[Cor. 3.5]{Ma24}}]\label{prop:hecke-vanish}
Fix a finite set $S$ of primes. Let $f\in\Mod_k(\rho_G)$ and let $\gamma,\mu\in G$ satisfy the following conditions:
    \begin{itemize}[leftmargin=2em]
        \item $\left(\prod\limits_{p\in S}p\right)\gamma=\left(\prod\limits_{p\in S}p\right)\mu$ and $\frq(\gamma)=\frq(\mu)$;
        \item for any $p\in S$, $\gamma,\mu$ satisfy conditions $(\ast_p)$.
    \end{itemize}
 Set
\[v_{\gamma,S}^\mu\coloneqq \sum_{I \subset S}(-1)^{|I|}\fre_{\gamma_{I}^{\mu}},\]
 where $\gamma_{I}^{\mu}\in G$ is the element whose $p$-adic component is equal to that of $\mu$ when $p\in I$ and other $p$-adic components remain the same as those of $\gamma$. Then $\prod\limits_{p\in S}\sum\limits_{n=1}^{\infty}\frac{\bT_{p^{2n}}}{p^{ns}}(f)=0$ implies that $\langle f, v^\mu_{\gamma,S}\rangle=0$.
\end{proposition}
\begin{proof}
By using Lemma \ref{lem:vec-sca}, one can verify the following formula
\begin{equation}\label{eq:inductive}
\begin{aligned}
     \langle \bT_{p^{2n}} (f), v^{\mu^{\gamma}_{\{p\}}}_{\gamma,S\backslash{\{p\}}} \rangle  & = \rT^{\frq(\gamma)}_{p^{2n}}\langle f,  v^{p^n\mu_{\{p\}}^{\gamma}}_{p^n\gamma,S\backslash{\{p\}}}\rangle,
\end{aligned} 
\end{equation}
and 
 \begin{equation}\label{eq:inductive2}
      \langle \bT_{p^{2n}} (f), v^\mu_{\gamma_{\{p\}}^\mu,S\backslash{\{p\}}} \rangle   = \rT^{\frq(\gamma)}_{p^{2n}}\langle f,  v^{p^n\mu}_{p^n\gamma^{\mu}_{\{p\}},S\backslash\{p\}}\rangle.
 \end{equation}
By induction on the cardinality of $S$, one can obtain the assertion by  \eqref{eq:inductive}, \eqref{eq:inductive2} and the fact
    \begin{equation}
        v^\mu_{\gamma, S}=v^{\mu^{\gamma}_{\{p\}}}_{\gamma,S\backslash{\{p\}}}-v^\mu_{\gamma_{\{p\}}^\mu,S\backslash{\{p\}}}.
    \end{equation}
\end{proof}

\section{Modular forms as theta series}\label{sec:mod-theta}\label{sec4}

This section studies the space of cusp forms $\Cusp_k(\rho_M)$ with respect to a positive definite even lattice $M$. Our goal is to show that under certain specific conditions on $\rho_M$,  the space $\Cusp_k(\rho_M)$ is spanned by the theta series. Throughout this section, $M$ denotes a positive definite even lattice of rank $r$.

\subsection{Projection to $\Cusp^\theta_k(\rho_M)$}

In \cite{EZ}, Eichler and Zagier introduced a differential operator transforming a $C^\infty$ function on $\HH_2$ into a $C^\infty$ function on $\HH\times\HH$. To define this operator, we need Gegenbauer polynomials $G^h_r(x,y)$, which are defined by the generating function:
\[\frac{1}{(1-2xT+yT^2)^{\frac{r}{2}-1}}=\sum_{h=0}^{\infty}G^h_r(x,y)T^h,\]
and $G_r^h(1,1)=\binom{r-3+h}{h}$.
The differential operator is then defined by:
\begin{align*}
    \partial_h: C^\infty(\HH_2)&\rightarrow C^\infty(\HH\times \HH)\\
    f&\mapsto G^h_r\left(\frac{1}{2}\frac{\partial}{\partial z_2},\frac{\partial^2}{\partial z_1\partial z_4}\right)f(Z)\mid_{z_2=0}, 
\end{align*}
where the coordinate $Z=\left(\begin{array}{cc}z_1 & z_2 \\z_2 & z_4\end{array}\right)\in\HH_2$. For the theta series, we have the following facts.
\begin{lemma}\label{lem:ortho}
Suppose $h>0$.  Let $\{H_i\}$ be an orthonormal basis for the space of harmonic polynomials of degree $h$. For $(z,z')\in\HH\times \HH\subseteq \HH_2$,  we have 
\begin{equation}\label{eq:theta-decom}
   \Theta^{(2)}_{M,1}(z,z')=\Theta^{(1)}_{M,1}(z)\otimes \Theta^{(1)}_{M,1}(z') 
\end{equation}
and 
\begin{equation}\label{eq:diff-theta-decom}
   ( \partial_h \Theta^{(2)}_{M,1})(z,z')=C \sum\limits_{i} \Theta^{(1)}_{M,F_i}(z)\otimes \Theta^{(1)}_{M,\overline{F}_i}(z')
\end{equation} 
for some non-zero constant $C$.
\end{lemma}
\begin{proof}
For the first assertion, this follows from the definition of the theta series. The second assertion then follows from \cite[Proposition 2.7]{Ma24}.

\end{proof}

Let $k=\frac{r}{2}+h$ for some non-negative integer $h$.  When $h=0$, assume further $r>4$. As in \cite{Ma24},  setting $\vartheta_{M,k}\coloneqq \partial_h \Theta_{\Gen(M)}^{(2)}$, we can define a theta lifting map 
\begin{equation}
  \Psi:  \Cusp_k(\rho_M)\to \Mod_k(\rho_M)
\end{equation}
by 
$$ \Psi(f)(z')= \int_{\SL_2(\ZZ)\backslash \HH} \langle f(z),\vartheta_{M,k}(z,-\overline{z'}) \rangle y^{k} \frac{\rmd x \rmd y}{y^2}$$
regarded as a function in $z'$. Then we have 
\begin{proposition}
Let $k, M$ be as above. The linear map $\Psi$ is diagonalizable and surjective onto $\Cusp^\theta_k(\rho_M)$. Consequently, there is a decomposition
\[\Cusp_k(\rho_M)=\Cusp^{\theta}_k(\rho_M)\oplus \ker\Psi.\]
\end{proposition}
\begin{proof}
    The proof is essentially \cite[Proposition 6.2]{Ma24}. According to Lemma \ref{lem:ortho}, $\Psi(f)(z')$ is equal to
 \begin{align*}
        &\int_{\SL_2(\ZZ)\backslash \HH}\left(\sum\limits_{L\in \Gen(M)} \sum\limits_{\sigma\in \mathrm{Iso}(G_M,G_L)} C_{L,\sigma} \langle f(z), \sigma^{\ast}\partial_h \Theta_{L,1}^{(2)}(z,-\overline{z'}) \rangle \right) y^{k} \frac{\rmd x \rmd y}{y^2} \\
        =&\sum\limits_{i}\sum\limits_{L\in \Gen(M)} \sum\limits_{\sigma\in \mathrm{Iso}(G_M,G_L)} C_{L,\sigma} \langle f,  \sigma^\ast \Theta_{L,F_i}^{(1)}\rangle_{\rm Pet} \cdot \sigma^\ast \Theta_{L,F_i}^{(1)}(z')
\end{align*}
    where the constant $C$ and the harmonic polynomials $\{H_i\}$ are those in Lemma \ref{lem:ortho}, and $$\Theta_{\Gen(M)}^{(2)}=\sum\limits_{L\in \Gen(M)} \sum\limits_{\sigma\in \mathrm{Iso}(G_M,G_L)} C_{L,\sigma} \sigma^\ast\Theta^{(2)}_{L,1} $$ is given in \eqref{eq:gen-theta}. When $h>0$, as $\Theta_{L,F_i}^{(1)}$ is cuspidal,  we have $\Psi(f)\in \Cusp^\theta_k(\rho_M)$. When $h=0$, we have all $H_i=1$. Note that $ \Theta_{\Gen(M)}^{(1)}=\bE^{(1)}_{k,M}$ (by Theorem \ref{thm:Eis=Theta}) is orthogonal to $f$ under the Petersson inner product, one can rewrite $\Psi(f)$ as 
 $$\sum\limits_{L\in \Gen(M)} \sum\limits_{\sigma\in \mathrm{Iso}(G_M,G_L)} C_{L,\sigma}\langle f, \sigma^{*}\Theta_{L,1}-\Theta^{(1)}_{\Gen(M)}\rangle_{\rm Pet}\cdot (\sigma^{*}\Theta_{L,1}-\Theta^{(1)}_{\Gen(M)}).$$  
Since $\sigma^\ast\Theta_{L,1}-\Theta^{(1)}_{\Gen(M)}$ is cuspidal,   the above formula implies $\Psi(f)\in \Cusp^\theta_k(\rho_M)$ as well. The assertion now follows from the Lemma  \ref{lemma-proj} below.     
\end{proof}
    \begin{lemma}[{\cite[Lem. 6.1]{Ma24}}]\label{lemma-proj}
    Let $V$ be a $\CC$-vector space with an inner product $\langle\cdot,\cdot\rangle$, and let $\{v_i\}_{i\in I}\subset V$ be an arbitrary finite family. Then the map $f:V\rightarrow V$ defined by $$f(v) = \sum_{i\in I}\langle v,v_i\rangle \cdot v_i$$ is diagonalizable and surjective onto $\mathrm{Span}\{v_i:i\in I\}$.
    \end{lemma}

Thus, a statement like "$\Cusp=\Cusp^{\theta}$" is equivalent to "$\ker\Psi=0$".

\subsection{$\Psi$ as Hecke operators}

Let $\alpha\in \ZZ_{>0}$. Define $K_\alpha(z,z'):\HH\times \HH\to \CC[G_M]\otimes \CC[G_M]$ by 
\begin{equation}
     K_\alpha(z,z')=\sum\limits_{\gamma\in G_M} \sum\limits_{\tg=(g,\phi_g(z))\in \widetilde{\bY}_{\alpha^2} } \frac{\phi_g(z)^{-2k}}{(z'+g\cdot z)^k}\rho_M(\widetilde{g})^{-1}(\fre_\gamma)\otimes \fre_\gamma.
\end{equation}
More explicitly, if $g=\begin{pmatrix}
    a & b\\ c& d
\end{pmatrix}$ and $\phi_g(z)=\pm\sqrt{cz+d}$, then 
\begin{equation}\label{eq:kern-eq1}
    \frac{\phi_g(z)^{-2k}}{(z'+g(z))^k}=\frac{(\pm 1)^{2k}}{(az+b+czz'+dz')^k}.
\end{equation}
As a function of $z$, one can check $\bT_{\alpha^2}(K_\alpha(z,z'))=\alpha^{2k-2}K_\alpha(z,z')$. As shown in \cite[Proposition 3.7]{Ma24} (see also Appendix \ref{appHK} for the odd signature case), it can be viewed as the kernel function of $\bT_{\alpha^2}$ up to a scalar, i.e.  it satisfies 
\begin{equation}\label{eq:ker-Heck}
    \frac{\bT_{\alpha^2}(f)(z')}{\alpha^{2k-2}}=C\int_{\SL_2(\ZZ)\backslash \HH}\langle f,  K_\alpha(z,-\overline{z}')\rangle y^{k-2}\rmd x \rmd y,
\end{equation}
for some non-zero constant $C$.

\begin{theorem}\label{thm:psi=hecke}
Assume $r=\rank(M)>6$, and let $k=\frac{r}{2}+h$ for some non-negative integer $h$. For any $f\in\Cusp_k(\rho_M)$, we have $$\Psi(f) =C'\sum\limits_{\alpha=1}^\infty \frac{1}{\alpha^{2k-2-h}}\bT_{\alpha^2}(f).$$
for some non-zero constant $C'$.
\end{theorem}

\begin{proof}
Let $\widetilde{C}\coloneqq h! \binom{r-3+h}{h} \binom{k-1}{h}$. For $\tau=(z,z')\in \HH\times \HH\subseteq \HH_2$, we claim that 
\begin{equation}\label{eq:dec}
    \vartheta_{M,k}(\tau) =\partial_h \Theta^{(2)}_{\Gen(M)} \\
    = \begin{cases}
       \displaystyle\bE_{k,M}^{(1)}(z)\otimes  \bE_{k,M}^{(1)}(z')+\frac{\bfe (\sign(M)/8)}{|G_M|^{\frac{1}{2}}}\sum\limits_{\alpha=1}^\infty K_\alpha(z,z')&\hbox{if $h=0$}\\  \\
 \displaystyle \widetilde{C}\frac{\bfe (\sign(M)/8)}{|G_M|^{\frac{1}{2}}}\sum\limits_{\alpha=1}^\infty \alpha^h K_\alpha(z,z') &\hbox{if $h>0$}
   \end{cases}
   \end{equation}
If the claim holds,  note that 
\[
    \int_{\SL_2(\ZZ)\backslash \HH} \langle f(z),  \bE_{k,M}^{(1)}(z)\otimes  \bE_{k,M}^{(1)}(-\overline{z'})) \rangle y^{k} \frac{\rmd x \rmd y}{y^2} \\
    =\langle f(z), \bE_{k,M}^{(1)}(z) \rangle_{\rm Pet}\cdot  \bE_{k,M}^{(1)}(z')=0\]
 for any $f\in \Cusp_k(\rho_M)$, then we have  
   \[\begin{aligned}
          \Psi(f)(z')&=\int_{\SL_2(\ZZ)\backslash \HH} \langle f(z),\vartheta_{M,k}(z,-\overline{z'}) \rangle y^{k} \frac{\rmd x \rmd y}{y^2}\\  
&= \widetilde{C}\frac{\bfe (-\sign(M)/8)}{|G_M|^{\frac{1}{2}}} \sum\limits_{\alpha=1}^\infty \int_{\SL_2(\ZZ)\backslash \HH} \langle f(z), \alpha^h K_\alpha(z,-\overline{z'}) \rangle y^{k} \frac{\rmd x \rmd y}{y^2}\\ & = \widetilde{C}\frac{\bfe (-\sign(M)/8)}{|G_M|^{\frac{1}{2}}} C^{-1}\sum\limits_{\alpha=1}^\infty \frac{1}{\alpha^{2k-2-h}}\bT_{\alpha^2}(f)(z'),
 \end{aligned}\]
by using \eqref{eq:dec} and \eqref{eq:ker-Heck}.

To prove the claim, let us first collect some facts. By Theorem \ref{thm:Eis=Theta}, we have $$\Theta_{\Gen(M)}^{(2)}=\bE^{(2)}_{\frac{r}{2},M}=\sum\limits_{\tg^{(2)}\in \widetilde{\Gamma}_\infty^{(2)}  \backslash \Mp_{4}(\ZZ)} \left(1\mid_{\frac{r}{2}}[\tg^{(2)}]\right)  \rho_{M}^{(2)}(\tg^{(2)})^{-1}(\fre_0\otimes \fre_0).$$ 

It's shown in Appendix \ref{appB} that for any coset $[\tg^{(2)}]\in \widetilde{\Gamma}_\infty^{(2)}  \backslash \Mp_{4}(\ZZ)$, it is uniquely determined by a pair  
     $$(\varphi(\tg^{(2)}), \alpha)\in {\bf \widetilde{Mat}}_2(\ZZ)\times \ZZ^{\leq 0},$$ 
(up to $\pm 1$ when $\alpha=0$),  where $\varphi(\tg^{(2)})=(g,\phi_g)$ satisfies $g \in \bY_{\alpha^2}$ and
\begin{equation}\label{eq:kern-eq}
    \frac{\phi_g(z)^{-2k}}{(z'+g(z))^k}=\phi_{g^{(2)}}(z,z')^{-2k}.
\end{equation}
Moreover, we can find a representative $\tg^{(2)}_0\in [\tg^{(2)}]$ such that 
\begin{equation}\label{eq:weil-decom}
 \rho^{(2)}_M(\tg^{(2)}_0)^{-1}(\fre_0\otimes \fre_0)=\frac{\bfe (\sign(M)/8)}{|G_M|^{\frac{1}{2}}}\sum\limits_{\gamma\in G_M} (\fre_{\gamma}\mid[\varphi(\tg^{(2)})])\otimes \fre_\gamma.
\end{equation}

As shown in the proof of \cite[Thm. 4.7]{Ma24},  recall $k=\frac{r}{2}+h$,   we have 
\begin{equation}\label{eq:diff}
    \begin{aligned}
    \partial_h\left(1\mid_{\frac{r}{2}}[\tg^{(2)}]\right)(\tau)&=  G_r^h(1,1)  |\alpha|^h \prod\limits_{i=0}^{h-1} \left(\frac{r}{2}+i\right) \left(1\mid_{k} [\tg^{(2)}] \right)(\tau) \\ &= \widetilde{C}  |\alpha|^h \left(1\mid_{k} [\tg^{(2)}] \right)(\tau).
\end{aligned}
\end{equation} 
for $\tau=(z,z')\in \HH\times \HH$, where the non-zero constant $\widetilde{C}=h! \binom{r-3+h}{h} \binom{k-1}{h}$.

Now we are ready to derive the formula. Let us first assume $h>0$. By \eqref{eq:diff}, \eqref{eq:weil-decom} and \eqref{eq:kern-eq}, we then have

\begin{equation*}
\begin{aligned}
    \partial_h\bE^{(2)}_{\frac{r}{2},M}(z,z')&= \sum\limits_{\alpha=0}^\infty \sum_{\substack{\tg^{(2)}\in \widetilde{\Gamma}_\infty^{(2)} \backslash \Mp_4(\ZZ)\\ \varphi(\tg^{(2)})\in \widetilde{\bY}_{\alpha^2} }} \partial_h \left(1\mid_{\frac{r}{2}}[\tg^{(2)}]\right)(z,z')  \rho_{M}^{(2)}(\tilde{g}^{(2)})^{-1}(\fre_0\otimes \fre_0)\\&= \widetilde{C}\sum\limits_{\alpha=1}^\infty \sum_{\varphi(\tg^{(2)})\in\widetilde{\bY}_{\alpha^2}} \alpha^h \left(1\mid_{k}[\tg^{(2)}]\right) (z,z')  \rho_{M}^{(2)}(\tilde{g}^{(2)})^{-1}(\fre_0\otimes \fre_0)  \\&=\widetilde{C} \frac{\bfe (\sign(M)/8)}{|G_M|^{\frac{1}{2}}}\sum\limits_{\alpha=1}^\infty \sum_{\varphi(\tg^{(2)})\in\widetilde{\bY}_{\alpha^2}} \alpha^h \left(1\mid_{k}[\tg^{(2)}]\right) (z,z')\left(\sum\limits_{\gamma\in G_M} (\fre_{\gamma}\mid[\varphi(\tg^{(2)})])\otimes \fre_\gamma\right)
    \\&= \widetilde{C} \frac{\bfe (\sign(M)/8)}{|G_M|^{\frac{1}{2}}} \sum\limits_{\alpha=1}^\infty \frac{1}{2} \sum_{\tg \in \widetilde{\bY}_{\alpha^2} }  \alpha^h\frac{\phi_g(z)^{-2k}}{(z'+g\cdot z)^k}\left(\sum\limits_{\gamma \in G_M} (\fre_{\gamma}\mid[\tg])\otimes \fre_\gamma\right)\\ &= \widetilde{C}\frac{\bfe (\sign(M)/8)}{|G_M|^{\frac{1}{2}}} \frac{1}{2}\sum\limits_{\alpha=1}^\infty \alpha^h K_\alpha(z,z').
\end{aligned}
\end{equation*}

When $h=0$, one has to deal with the determinant zero part and show that  $$\frac{\bfe (\sign(M)/8)}{|G_M|^{\frac{1}{2}}}\frac{1}{2}\sum\limits_{\tg\in \widetilde{\bY}_0/\pm 1} \frac{\phi_g(z)^{-2k}}{(z'+g\cdot z)^{k}} \sum\limits_{\gamma \in G_M}(\fre_{\gamma}\mid[\tg])\otimes \fre_\gamma=\bE_{k,M}(z)\otimes \bE_{k,M}(z').$$  The argument is essentially the same (see \cite[Thm. 4.6]{Ma24}) and  we omit the details. 

\end{proof}

\subsection{Injectivity of the projection}
 
\begin{theorem}\label{thm-cusp=theta}
 Assume that $M$ is a positive definite even lattice of level $N$ and rank $r>6$. Furthermore, assume that for all primes $p$, $M\otimes \ZZ_p$ splits a hyperbolic plane.   For  $k \geq  r/2$, we have $$\Cusp^\theta_k(\rho_{M})=\Cusp_k(\rho_{M}).$$
\end{theorem}
\begin{proof}
As mentioned before, it suffices to show that the linear map $\Psi$ is injective.  When $\sign(M)$ is even, this is exactly \cite[Thm. 6.7]{Ma24}. When $\sign(M)$ is odd, the argument is very similar, and we may sketch the proof for the ease of readers. 

Suppose $f\in\Cusp_k(\rho_{M})$ satisfies $\Psi(f)=0$. By Theorem \ref{thm:psi=hecke},
\begin{align*}
        \Psi(f) &=C'\sum\limits_{\alpha=1}^\infty \frac{\bT_{\alpha^2}}{\alpha^{2k-2-h}}(f) \\
        &=C'\left(\sum_{\substack{\alpha\geq 1 \\ \gcd(\alpha, N)=1}}\frac{\bT_{\alpha^2}}{\alpha^{2k-2-h}}\right)\left(\prod\limits_{p\mid N} \sum_{n=0}^{\infty} \frac{\bT_{p^{2n}}}{p^{n(2k-2-h)}} \right)(f)
    \end{align*}
    for some non-zero constant $C'$. By Corollary \ref{cor:basis}, 
    choose a basis of $\Cusp_k(\rho_{M})$ consisting of simultaneous eigenforms for all Hecke operators in $\{\bT_{\alpha^2}:\gcd(\alpha, N)=1\}$. Now by Theorem \ref{thm:nonvanish}, for each element $g$ in this basis,
    \[\left(\sum_{\substack{\alpha\geq 1 \\ \gcd(\alpha, N)=1}}\frac{\bT_{\alpha^2}}{\alpha^{2k-2-h}}\right)(g)=L(g,2k-2-h)g\neq 0.\]
    Therefore the operator $\sum\limits_{\substack{\alpha\geq 1 \\ \gcd(\alpha, N)=1}}\frac{\bT_{\alpha^2}}{\alpha^{2k-2-h}}$ is invertible, which implies
    \[ \prod\limits_{p\mid N}\sum\limits_{n=1}^{\infty}\frac{\bT_{p^{2n}}}{p^{n(2k-2-h)}}(f)=0. \]
        
Let $n\coloneqq \prod\limits_{p\mid N}p$ be the radical of $N$. By \cite[Lem. 6.6]{Ma24}, according to our assumption on $M$, there exists a sublattice $L\subseteq M$ such that $G_L\cong G_{M}\oplus G'$, where $G'=\left(\begin{array}{cc}0 & \frac{1}{n} \\\frac{1}{n} & 0\end{array}\right) \bmod \ZZ$ under a basis $\{x_1,x_2\}$ and $nx_i=0$.  Then the cyclic subgroup $H=\langle x_1\rangle$ is an isotropic subgroup of $G_L$ and $H^\perp/H\cong G_{M}$.

For any $\gamma \in G_{M}$,  we have 
    \begin{itemize}
    
        \item $x_1 + \gamma, x_2+\gamma$ are not divisible by $p$ in $G_L$ for any $p\mid N$.
        \item when $l$ is even, then
$$2 \frq\left(\frac{nx_1}{2}\right) + \left(\frac{n}{2}x_1,x_2 + \gamma\right) = 2 \frq\left(\frac{nx_2}{2}\right) + \left(\frac{n}{2}x_2,x_1 + \gamma \right)=\frac{1}{2}\neq 0 \mod \ZZ.$$
    \end{itemize}

According to \cite[Lem. 6.3]{Ma24}, also $\prod\limits_{p\mid N}\sum\limits_{n=1}^{\infty}\frac{\bT_{p^{2n}}}{p^{n(2k-2-h)}}\uparrow_{H}(f)=0$. Let $v=v_{x_1+\gamma, x_2+\gamma, S}$ and  then 
\begin{equation*}
    \begin{aligned}
    \downarrow_H(v)&=\sum\limits_{I\subset\{p\mid N\}} (-1)^{|I|} \downarrow_H \fre_{(x_1+\gamma)_{I}^{{x_2+\gamma}}} \\ &=\fre_\gamma
\end{aligned} \end{equation*}  as $(x_1+\gamma)_{I}^{x_2+\gamma}$ is not orthogonal to $H$ unless $I=\varnothing$. Then  
\[\langle f,\fre_\gamma\rangle =\langle \uparrow_H (f), v\rangle = 0,\]
where we used Proposition \ref{prop:hecke-vanish} for the last equality. Since $\gamma$ is arbitrary, this implies that $f=0$.
\end{proof}

\begin{remark}\label{rem:local split}
There is a well-established criterion for determining whether $M\otimes \ZZ_p$ splits a hyperbolic plane, as mentioned in \cite{Ni79}. For example, this occurs if $\rank_p(G_M)<r-2$. In practical terms, this condition is not overly restrictive when $r$ is sufficiently large, as $\rank_p(G_M)\leq r$ always holds.
\end{remark}

\begin{remark}\label{rem:localConditionNecessary}
When $M \otimes \mathbb{Z}_p$ does not split a hyperbolic plane, $\Cusp_k(\rho_M)$ may strictly contain $\Cusp^\theta_k(\rho_M)$. Consider for example $M=L(N)$ where $L$ is unimodular of rank $m$ and $N$ is a square-free integer. For sufficiently large $m$ and $N$, there exist non-zero newforms $f$ of weight $m/2$ for $\Gamma_0(N)$. 

The zero-component of any element in $\Cusp^\theta_k(\rho_M)$ is an oldform. However, by \cite[Thm. 1.1]{SV}, for any such newform $f$, the lift
\[
\sum_{g \in \Gamma_0(N)\backslash \mathrm{SL}_2(\mathbb{Z})} (f \cdot \mathfrak{e}_0)|_{k, G_M} [g]
\]
has a non-zero newform as its zero-component, thus cannot lie in $\Cusp^\theta_k(\rho_M)$.
\end{remark}

\section{Heegner divisors on orthogonal Shimura varieties}\label{sec5}

In this section, we review the general theory of Heegner divisors on orthogonal Shimura varieties and their connections with modular forms.  We develop this correspondence to the Baily--Borel compactifications.

\subsection{Orthogonal Shimura variety and its compactification}
Let \( M \) be an even lattice of signature \((2,n)\).
The Hermitian symmetric domain of type IV associated to \( M \) is
\[
\bD^\pm = \left\{ w \in \PP(M \otimes \CC) \mid \langle w, w \rangle = 0,\, \langle w, \overline{w} \rangle > 0 \right\},
\]
with two connected components. Fix a component \(\bD\subset \bD^\pm\). We denote by \(\rO(M)(\RR)^+\leq \rO(M)(\RR)\) the index-2 subgroup preserving \(\bD\). Given any arithmetic subgroup \(\Gamma \leq \rO(M)(\RR)^+\), we define the orthogonal Shimura variety
\[
\Sh_\Gamma(M) \coloneqq \Gamma \backslash \bD.
\]
\begin{remark}\label{rem:Spinorkernel}
    Note that \(\Gamma^{\pm}\backslash\bD^\pm\cong\Gamma\backslash\bD\) for any \(\Gamma^{\pm} \leq \rO(M)(\RR)\) such that \(\Gamma = \Gamma^{\pm}\cap\rO(M)(\RR)^+\) has index 2 in \(\Gamma^{\pm}\).
\end{remark}

Let $\Pic_\QQ(\Sh_\Gamma (M))$ denote the Picard group of $\Sh_\Gamma(M)$ with rational coefficients. 
Since $\Sh_\Gamma(M)$ has only quotient singularities, it is $\QQ$-factorial, and we have an isomorphism $$\Pic_\QQ(\Sh_\Gamma(M))\cong \Cl_\QQ(\Sh_\Gamma(M)).$$
The tautological line bundle \(\cO(-1)\) on \(\bD\subset \PP(M \otimes \CC)\) descends to an element $\lambda_{M,\Gamma}$ in $\Pic_\QQ(\Sh_\Gamma(M))$, called the \textit{Hodge bundle}.  The \textit{Baily--Borel compactification} \(\overline{\Sh}_\Gamma(M)\) is constructed as 
\[
\overline{\Sh}_\Gamma(M) = \Proj\left( \bigoplus_{m\geq 0} \rH^0(\Sh_\Gamma(M), \lambda_{M,\Gamma}^{\otimes m}) \right),
\]
endowing \(\lambda_{M,\Gamma}\) with an ample extension \(\overline{\lambda}_{M,\Gamma}\). Usually, we will simply write $\lambda_{M,\Gamma}$ if there is no further confusion.

\subsubsection*{\bf Explicit Boundary Structure}  The compactification admits a stratification:
\[
\overline{\Sh}_\Gamma(M) = \Sh_\Gamma(M) \sqcup \bigsqcup_{I} \partial_I \sqcup \bigsqcup_{J} \partial_J,
\]
where:
\begin{itemize}
\item \(I\) ranges over \(\Gamma\)-orbits of primitive isotropic lines in \(M\)
\item \(J\) ranges over \(\Gamma\)-orbits of primitive isotropic planes in \(M\)
\end{itemize}
Each boundary divisor \(\partial_J\) corresponds to a \(\Gamma\)-orbit of isotropic planes \(J \subset M\), with the projection 
\[
\pi_{J^\perp}: \PP(M\otimes \CC) \setminus \PP(J^\perp \otimes \CC) \to \PP((M/J^\perp)\otimes \CC)
\]
inducing the boundary map \(\pi_{J^\perp}(\bD) \hookrightarrow \partial_J\).

In this paper, we primarily investigate two key cases that are closely related to moduli theory:
\begin{enumerate}
    \item $\Gamma$ is a subgroup of the \textit{stable orthogonal group} \[\widetilde{\rO}(M) \coloneqq \ker\left(\rO(M) \to \rO(G_M)\right) \cap \rO(M)(\RR)^+.\] 
    When $\Gamma=\widetilde{\rO}(M)$, we write $\Sh_\Gamma(M)$ as $\Sh(M)$.  Many  moduli space of \textit{lattice-polarized hyper-K\"ahler varieties} arises in this way. 
    
    \item  \(\Gamma=\Gamma_0:=\rO(M) \cap \rO(M)(\RR)^+\), the full orthogonal group. This case governs a lot of moduli spaces of \textit{bounded Calabi--Yau pairs} $(X,D)$.
     \end{enumerate}

\subsection{Heegner Divisors on $\Sh_\Gamma(M)$}
We begin by recalling the  construction of special divisors. For any vector $v \in M^\vee \subseteq M \otimes \QQ$ with $\frq_M(v) < 0$, the \emph{natural divisor} $H_v$ on $\Sh_\Gamma(M)$ is defined as the image of the natural map
\begin{equation}\label{eq:div-map}
  \Gamma_v \backslash \bD(v^\perp) \hookrightarrow \widetilde{\rO}(M) \backslash \bD,
\end{equation}
where $\Gamma_v \subseteq \widetilde{\rO}(M)$ denotes the stabilizer subgroup of $v$.

\begin{definition}
For a subgroup $\Gamma \leq \widetilde{\rO}(M)$, $m \in \QQ_{\leq 0}$, and $\gamma \in G_M$ with $(m,\gamma) \neq (0,0)$, the \emph{Heegner divisor} is defined as
\[
\rH_{m,\gamma} = \Gamma \Bigg\backslash \sum_{\substack{v \in M + \gamma \\ \frq_M(v) = m}} v^\perp
\]
on $\Sh_\Gamma(M)$. Conventionally set $\rH_{0,0} = -\lambda_{M,\Gamma}$ when $(m,\gamma) = (0,0)$. The subspace generated by Heegner divisors is denoted by $\Pic_\QQ(\Sh_\Gamma(M))^{\mathrm{Heegner}}$.
\end{definition}
The Heegner divisor $\rH_{m,\gamma} \in \Cl_\QQ(\Sh_\Gamma(M))$ is generally reducible, and its irreducible components are natural divisors.  In cases where $\gamma = -\gamma$ in $M^\vee/M$, the divisor $\rH_{m,\gamma}$ acquires multiplicity two.

A fundamental property is the compatibility of Heegner divisors with pullback: for subgroups $\Gamma' \leq \Gamma$, the natural projection $\pi: \Sh_{\Gamma'}(M) \to \Sh_{\Gamma}(M)$ satisfies  
\[
\pi^* \rH_{m,\gamma} = \rH_{m,\gamma}^{\Gamma'},
\]  
where $\rH_{m,\gamma}^{\Gamma'}$ denotes the Heegner divisor on $\Sh_{\Gamma'}(M)$. Consequently, $\pi^*$ induces an isomorphism  
\[
\Pic_\QQ(\Sh_\Gamma(M))^{\mathrm{Heegner}} \cong \Pic_\QQ(\Sh_{\Gamma'}(M))^{\mathrm{Heegner}}.
\]  
In particular, the space $\Pic_\QQ(\Sh_\Gamma(M))^{\mathrm{Heegner}}$ is independent of the choice of $\Gamma$.

\begin{example}\label{ex:Fg}
A crucial example is the Shimura varieties associated with the lattice
\[
\Lambda_g \coloneqq \langle 2 - 2g \rangle \oplus E_8(-1)^{\oplus 2} \oplus U^{\oplus 2},
\]  
The global Torelli theorem for K3 surfaces provides a period map isomorphism  
\[
\scF_g \cong \Sh(\Lambda_g),
\]  
establishing a correspondence between Noether--Lefschetz divisors on \(\scF_g\) and Heegner divisors on \(\Sh(\Lambda_g)\).  
\end{example} 

For a primitive vector $u\in M$, we call the positive generator $d$ of the ideal $\langle u,M\rangle$ the \textit{divisibility} of $u$ and set $u^* \coloneqq u/d+M\in G_M$. A basic result is 
\begin{theorem}[cf.~{\cite[Thm. 1.11]{BLMM17}}]
If $\dim \Sh_\Gamma(M)\geq 3$, then $\Pic_\QQ(\Sh_\Gamma(M))$ is spanned by natural divisors.
Furthermore, if $M$ contains two hyperbolic planes and $\Gamma=\widetilde{\rO}(M)$, then \[\Pic_\QQ(\Sh(M))=\Pic_\QQ(\Sh(M))^{\rm Heegner}. \]
\end{theorem}
\begin{proof}
 For the ease of readers, we recall the proof of the last assertion.    When $M$ contains two hyperbolic planes, Eichler's criterion says two primitive vectors $u$ and $v$ are lying in the same $\widetilde{\rO}(M)$-orbit if $u^2=v^2$ and $u^\ast=v^\ast$. This implies the primitive Heegner divisor
 \begin{equation}\label{eq:prim-heegner}
\rP_{m,\gamma}: = \widetilde{\rO}(M)\backslash  \sum_{\substack{
    v \in M \text{ primitive} \\
    \frq_M(v) = m \\
    v^\ast \equiv \gamma \pmod{M}
  }} \!\!\! v^\perp
\end{equation}
is irreducible.  This shows that $\Pic_\QQ(\Sh(M))$ is generated by the primitive divisors $\rP_{m,\gamma}$.  As there is a triangular relation between the Heegner divisors $\rH_{m,\gamma}$ and the primitive Heegner divisors,  it follows that $\Pic_\QQ(\Sh(M))$ is spanned by $\rH_{m,\gamma}$ as well. 

\end{proof}

When $M$ lacks two hyperbolic planes, Eichler's criterion generally fails. It remains unknown whether $\Pic(\Sh(M))$ is spanned by Heegner divisors.   Nevertheless, if we consider the $\rO(M)$-orbits,  there is a similar result for $p$-elementary lattices. 

\begin{proposition}[generalization of Eichler’s criterion]
Let $p$ be a prime number. Let $M$ be a $p$-elementary even lattice of signature $(2,n)$ which satisfies one of the following conditions:
\begin{itemize}
\item[(i)] $p \neq 2$ and $n>2$;
\item[(ii)] $p = 2$ with $\ell(G_M) \leq n $;
\end{itemize}
For any primitive vectors $u,v \in M$ satisfying
\begin{equation}\label{eq:numerical}
    u^2 = v^2\quad \hbox{and}\quad u^* = v^* \in G_M,
\end{equation} there exists $g \in \Gamma_0$ such that $g(u) = v$.

\end{proposition}
\begin{proof}

The case $p > 2$ follows from \cite[Thm. 5.15]{BC23}. For $p = 2$, we analyze two cases:

\medskip\noindent
\textbf{Case 1: $\ell(G_M)\leq n-2$}\\
In this case $M$ contains two hyperbolic planes. Eichler's criterion directly applies since $U^{\oplus 2} \subset M$ guarantees transitivity on primitive vectors with matching invariants.

\medskip\noindent
\textbf{Case 2: $ \ell(G_M)=n$}\\
We can decompose $M\cong U\oplus U(2)\oplus L(2)$ where $L=0$ or a negative definite unimodular lattice. Our argument depends on $L$'s parity. When $L$ is even, this can be proved by using Allcock's trick (cf. \cite[Lem. 1]{Allcock99}). See also  \cite[Thm.  2.3.2]{KocaCaner2005OitA} for a complete proof.
If $L$ is odd,  then there is a further decomposition (cf. \cite[Thm. 3.6.2]{Ni79})  $$M \cong U \oplus U(2)\oplus A_1(-1)^{\oplus m}.$$ We may fix the standard basis 
\begin{itemize}
\item $e_1, f_1$ for $U$ with $\langle e_1,f_1 \rangle = 1$,
\item $e_2, f_2$ for $U(2)$ with $\langle e_2,f_2 \rangle = 2$,
\item $\{r_i\}_{1\leq i\leq m}$ for $A_1(-1)^{\oplus m}$ with $r_i^2 = -2$.
\end{itemize}

Let \(u, v \in M\) be primitive with \(u^2 = v^2\) and \(u^* = v^* \in G_M\). Let \(d\) be their divisibility.  Set \(w = (u - v)/d\in M\) and make the decomposition:
\[
u = u_0 + u_R, \quad v = v_0 + v_R, \quad u_0, v_0 \in U \oplus U(2), \quad u_R, v_R \in A_1(-1)^{\oplus m}.
\]
Then we are led to analyze the following subcases.
\vspace{.2cm}

\paragraph*{\bf Subcase (i): \(u_0^* \neq 0\), \(v_0^* \neq 0\)}
As mentioned before, there exist \(h, h' \in \mathrm{O}(U \oplus U(2))\) such that \(h(u_0), h'(v_0) \in U(2)\oplus A_1(-1)^{\oplus m}\). Extend to \(M\) via block matrices:
\[
H = \begin{pmatrix}
h & 0 \\
0 & I_{m}
\end{pmatrix}, \quad H' = \begin{pmatrix}
h' & 0 \\
0 & I_{m}
\end{pmatrix},
\]
where \(I_{m}\) is the identity on \(A_1(-1)^{\oplus m}\). Apply \(H\) and \(H'\) to assume \(u_0, v_0 \in U(2)\). Thus, \(u, v \in U(2) \oplus A_1(-1)^{\oplus m}\).

Choose \(u', v' \in U(2)\oplus  A_1(-1)^{\oplus m}\) with \(\langle u, u' \rangle = d\), \(\langle v, v' \rangle = d\). Define Eichler transvections for isotropic \(x\) and \(y \perp x\):
\[
t(x, y)(\nu) = \nu - \langle y, \nu \rangle x + \langle x, \nu \rangle y - \frac{\langle x, \nu \rangle \langle y, y \rangle}{2} x,
\]
where $x,y\in M\otimes \QQ$.  The composition
\[
g = t(e_1, -v') \circ t(f_1, w) \circ t(e_1, u')
\]
satisfies \(g(u) = v\).
 \vspace{.2cm}
 
\paragraph{\bf Subcase (ii): \(u_0^* = 0\), \(v_0^* = 0\)}  Similarly, assume \(u_0, v_0 \in U\) via transformations as in (i). Thus, \(u, v \in U \oplus A_1(-1)^{\oplus m}\). Take $u',v'\in U\oplus A_1(-1)^{\oplus m}$ with $\langle u,u'\rangle =\langle v,v'\rangle=d $. 
The composition
\[
g' = t(e_2, -v') \circ t(f_2, \frac{w}{2}) \circ t(e_2, u')
\]
sends $u$ to $v$.  The problem is that we have to show \(t(f_2, w/2)\) is an integral transformation. 
This requires:
\begin{itemize}[leftmargin=0.5cm]
    \item [-] the divisibility of $w$ is even. 
\item [-] \(4 \mid w^2\).
\end{itemize}

If \( w \) does not satisfy these conditions, we claim that there exists \(\phi \in \mathrm{O}(U \oplus A_1(-1)^{\oplus m})\) such that \((\phi(u) - v)/d\) satisfies the desired conditions. The transformation \(\phi\) is achieved by possibly composing the following transformations:

\begin{enumerate}[leftmargin=*]
    \item The swap \(\sigma: e_1 \leftrightarrow f_1\). 

    \item For each \(i\), the Eichler transvection \(t(-e_1,r_i)\), which acts on basis as:
   \[
   t(-e_1,r_i): \begin{cases}
   e_1 \mapsto e_1 \\
   f_1 \mapsto f_1 + e_1 -r_i \\
   r_i \mapsto r_i - 2e_1 \\
   r_j \mapsto r_j \quad \forall j \neq i
   \end{cases}
   \]
    \item Sign changes \(r_i \mapsto -r_i\) for indices with odd coefficients.
\end{enumerate}
Here, we use  \(\sigma\) and \( t(-e_1,r_i) \) to ensure that the divisibility of \( w \) is even, while the sign change mapping ensures that \( 4 \mid w^2 \). Note that these transformations do not change the dual of these vectors in \( G_M \).
\vspace{.2cm}

\paragraph{\bf Subcase (iii): \(u_0^* = 0\), \(v_0^* \neq 0\)} Here \(d = 2\). One can use the orthogonal transformations in (i)  and (ii) to normalize our vectors without changing the norm and the dual in $G_M$. After normalization, assume:
\[
u_0 = 2e_1 + u_1, \quad u_1 \in U(2), \quad u_R = r_i+\sum\limits_{j\neq i} c_j r_j, \quad v_0 \in U(2),
\]
for some $i$.
As the transformation $ t(-e_1,r_i)$ maps $(2e_1+r_i)$ to $r_i$,   it sends $u$ in \( U(2) \oplus  A_1(-1)^{\oplus m}\).  Then proceed as in (i) to map \(t(-e_1,r_i)(u)\) to \(v\).

\end{proof}

As an immediate consequence (together with Remark \ref{rem:Spinorkernel}), we obtain the following.

\begin{corollary}\label{cor:Heeg-span}
    Let $\pi:\Sh(M)\to \Sh_{\Gamma_0}(M)$ be the quotient map.  The image of the pullback map 
 \begin{equation}
     \pi^\ast: \Pic_\QQ(\Sh_{\Gamma_0}(M))\to \Pic_\QQ(\Sh(M))
 \end{equation}
 is contained in $\Pic_\QQ(\Sh(M))^{\rm Heegner}$.
\end{corollary}

\subsection{Heegner divisors as Coefficient Functionals}

\begin{definition}
  Define the space of almost cusp forms of weight $k$ and type $\rho_{M}^*$ 
$$\mathrm{ACusp}_k(\rho_{M}^* )\subseteq \Mod_k(\rho_{M}^*) $$  
as the subspace generated by cusp forms and the Siegel--Eisenstein series $E_{k,M}$.  
\end{definition}
 Then we have

\begin{theorem}[\cite{Br14}]\label{thm:NLC}
Assume that $n\geq 3$ and $M$ contains $U\oplus U(N)$ as a direct summand. There is an isomorphism 
\begin{equation*}
 \Phi:   \Pic_\CC(\Sh(M))^{\rm Heegner}\cong \mathrm{ACusp}_{\frac{2+n}{2}}(\rho_M^*)^\vee
\end{equation*}
realized by the \textbf{coefficient functionals}:
\begin{equation}\label{eq:coe-funct}
\rH_{m,\gamma} \longmapsto \Bigg( \underbrace{f = \sum_{\gamma'\in G_M} \sum_{m'} c_{-m',\gamma} q^{m'} \fre_{\gamma'}}_{\text{Almost cusp form}} \ \mapsto \ \boxed{c_{-m,\gamma}} \Bigg).
\end{equation}
\end{theorem}

\begin{remark}
    The assertion also holds for $M$ that does not split a hyperbolic plane. See \cite[Thm. 1.4]{Br14} and \cite[Thm. 1.1]{Stein23}.
\end{remark}

The geometric-modular correspondence established in Theorem \ref{thm:NLC} yields concrete criteria for detecting special divisors. A prime example is the characterization of the Hodge bundle through cusp form annihilation:

\begin{theorem}[Hodge Bundle Criterion]\label{thm-heegner=hodge}
Under the assumptions of Theorem~\ref{thm:NLC}, a finite linear combination $$\cH=\sum\limits_{\gamma\in G_M} \sum\limits_{m\in \QQ} a_{m,\gamma}\rH_{m,\gamma}$$ is proportional to  the Hodge line bundle $\rH_{0,0}=-\lambda_M$ if  and only if  its associated coefficients  functionals \textit{annihilate all cusp forms}, i.e. 
 \begin{equation}\label{eq:vanishing}
     \sum\limits_{\gamma\in G_M} \sum\limits_{m\in \QQ}  a_{m,\gamma} c_{-m,\gamma}=0
 \end{equation}
 for all cusp forms $ \sum\limits_{\gamma\in G_M} \sum\limits_{m\in \QQ}  c_{m,\gamma}q^m\fre_\gamma \in \mathrm{Cusp}_{\frac{2+n}{2}}(\rho_M^*)$. 
\end{theorem}
\begin{proof}
If $\cH$ satisfies \eqref{eq:vanishing}, then the functional
$\Phi(\cH)\in \mathrm{ACusp}_{\frac{2+n}{2}}(\rho_M^*)^\vee$ annihilates all cusp forms, i.e.:
\[\Phi(\cH)(f)=0, \forall f\in \mathrm{Cusp}_{\frac{2+n}{2}}(\rho_M^*). \]
Thus, $\Phi(\cH)$ is proportional to $\Phi(\rH_{0,0})$, the functional that extracts the constant term. Hence, $\cH$ is a multiple of $\rH_{0,0}$. The converse follows directly. 
\end{proof}

\subsection{Extending Heegner Divisors to the Boundary}
We mainly focus on the case $\Sh(M)$. 
Let \( x \in \partial \overline{\Sh}(M) \) be a boundary point. The analytic local Picard group of \( \overline{\Sh}(M) \) at \( x \) is defined as the direct limit
\begin{equation*}
    \Pic_\QQ(\overline{\Sh}(M),x) \coloneqq \lim_{\substack{\longrightarrow \\ x \in U}} \Pic_\QQ(U \cap \Sh(M)),
\end{equation*}
where \( U \) ranges over analytic open neighborhoods of \( x \). The natural restriction map
\begin{equation}
    \Pic_{\QQ}(\Sh(M)) \to \Pic_{\QQ}(\overline{\Sh}(M),x)
\end{equation}
relates global and local Cartier divisor classes.

 A point \( x \in \partial \overline{\Sh}(M) \) is called \textit{generic} if it is not a cusp. Given an isotropic plane \( J \subset M \), we say a divisor \( H \in \Pic_\QQ(\Sh(M)) \) is \textit{trivial at generic points of \( \partial_J \subseteq  \overline{\Sh}(M) \)} if there exists a generic point \( x \in \partial_J  \) where the image of \( H \) under the restriction map vanishes in \( \Pic_\QQ(\overline{\Sh}(M),x) \). 
 
 In \cite{BF01}, Bruinier and Freitag  developed a local analogue of Borcherds products to characterize when a linear combination of Heegner divisors vanishes at generic boundary points. We will restate their criterion in a form adapted to our setting. In the following, we assume the lattice $M$ has  the  \textbf{decomposition} 
     \begin{equation}
         M = U(N_1) \oplus U(N_2) \oplus L
     \end{equation}
     where $L$ is an even definite lattice, \( N_1, N_2 \in \ZZ \setminus \{0\} \). This decomposition induces a decomposition of discriminant groups
\[G_M= G_{U(N_1)\oplus U(N_2)}\oplus G_L.\]
Let \( e_i, f_i \in U(N_i) \) be the standard basis of $U(N_i)$.  Let \( J \subset M \) be the standard isotropic plane generated by \( e_1 \) and \( e_2 \), 
and define the isotropic subgroup \( H_J \subset G_M \) as the span of \( e_1/N_1 \) and \( e_2/N_2 \). Then the composition
\begin{equation}\label{decom-iso}
    G_L\subset H_J^\perp\twoheadrightarrow H_J^\bot/H_J
\end{equation}
is an isomorphism. Under this isomorphism, we let 
\( \pi: H_J^\bot \twoheadrightarrow G_L \) denote the projection map induced by $H_J^\perp\twoheadrightarrow H_J^\bot/H_J$. The following result is a specialization of one direction of \cite[Thm. 5.1]{BF01}:

\begin{theorem}\label{thm-BF01}
    Let \( x \in \partial_J  \) be a generic point. If a finite linear combination 
    \[
        \sum_{\gamma \in G_M} \sum_{m \in \QQ} a_{m,\gamma} \rH_{m,\gamma}
    \]
    has trivial image in \( \Pic_{\QQ}(\overline{\Sh}(M),x) \), then for every isotropic-lifted theta series 
    \[
        \uparrow_{H_J}(\Theta_{L,F}) = \sum_{\gamma \in G_M} \sum_{m \in \QQ} c_{m,\gamma} q^m \mathbf{e}_\gamma,
    \]
    where \( F \) ranges over degree 2 harmonic polynomials, we have
    \[
        \sum_{\gamma \in G_M} \sum_{m \in \QQ} a_{m,\gamma} c_{-m,\gamma} = 0.
    \]
\end{theorem}

\begin{proof}
By \cite[Lem. 4.1]{BF01} and subsequent remarks, the image of \( \rH_{m,\gamma} \) in \( \Pic_{\QQ}(\overline{\Sh}(M),x) \) vanishes when \( \gamma \notin H_J^\perp \). Therefore, the triviality condition reduces to 
    \[
        \sum_{\gamma \in H_J^\perp} \sum_{m \in \QQ} a_{m,\gamma} \rH_{m,\gamma} = 0 \quad \text{in } \Pic_{\QQ}(\overline{\Sh}(M),x).
    \]
    Applying \cite[Thm. 5.1]{BF01} to this reduced sum, we obtain
    \begin{equation*}
        \sum_{\gamma \in H_J^\perp} \sum_{m \in \QQ} a_{m,\gamma} c'_{-m, \pi(\gamma)} = 0
    \end{equation*}
    for all theta series 
    \[
        \Theta_{L,F} = \sum_{\gamma \in G_L} \sum_{m \in \QQ} c'_{m,\gamma} q^m \mathbf{e}_\gamma,
    \]
    where \( F \) runs over degree 2 harmonic polynomials.

    By the definition of isotropic lifting, \( \uparrow_{H_J}(\Theta_{L,F}) \) satisfies \( c_{m,\gamma} = c'_{m, \pi(\gamma)} \) if \( \gamma \in H_J^\perp \), and \( c_{m,\gamma} = 0 \) otherwise. Thus, the above equality extends to
    \[
        \sum_{\gamma \in G_M} \sum_{m \in \QQ} a_{m,\gamma} c_{-m,\gamma} = 0,
    \]
    completing the proof.
\end{proof}

\subsection{An obstruction space}\label{sec-obs}
Based on Theorem~\ref{thm-BF01}, we introduce a subspace of $\Cusp_k(\rho_M^*)$ capturing obstructions to extending Heegner divisor classes to the boundary. 
Let $M$ be an even lattice admitting an orthogonal decomposition
\begin{equation}\label{admissible}
  M = U(N_1) \oplus U(N_2) \oplus L,
\end{equation}
where:
\begin{itemize}
  \item $L$ is a negative definite even lattice;
  \item $N_1, N_2 \in \mathbb{Z}_{\neq 0}$.
\end{itemize}
 We say $J\subset M$ is an \textit{admissible} isotropic plane, if $J$ is the standard isotropic plane in $U(N_1) \oplus U(N_2)\subset M$ for some decomposition of $M$ as \eqref{admissible}. Let $H_J \subseteq G_M$ be the associated isotropic subgroup. 
\begin{definition}
The obstruction space is defined as 
  \[
  \Cusp^{\mathrm{obs}}_{\frac{n+2}{2}}(\rho_{M}^*) \coloneqq \mathrm{Span}\left\{\, \uparrow_{H_J}(\sigma^*\Theta_{J^\perp/J,F}) \ \bigg| \ 
  \begin{aligned}
    & J \subset M \hbox{~admissible}, \\
    & F ~\hbox{harmonic of degree 2}, \\
    & \sigma \in \rO(H_J^\bot/H_J)
  \end{aligned} \right\}.
  \]
\end{definition}

Let $\Pic_\QQ(\overline{\Sh}(M))^{\rm Heegner}\subseteq \Pic_\QQ(\overline{\Sh}(M))$ be the subspace  consisting of elements whose restriction to $\Sh(M)$ is lying in $\Pic_\QQ(\Sh(M))^{\rm Heegner}$. A key observation is 

\begin{proposition}\label{prop:obs=pic1}
Let $M$ be an even lattice as above, and further assume $M$ splits a hyperbolic plane.  If \[\Cusp_{\frac{n+2}{2}}^{\rm obs}(\rho_M^\ast)=\Cusp_{\frac{n+2}{2}}(\rho_M^\ast),\] then 
 $\Pic_\QQ(\overline{\Sh}(M))^{\rm Heegner}\cong \QQ$
 is spanned by the extended Hodge bundle. 
\end{proposition}

\begin{proof}
As $\overline{\Sh}(M)$ is normal, one can view $\Pic_\QQ(\overline{\Sh}(M))^{\rm Heegner}$ as a subspace of $\Pic_\QQ(\Sh(M))^{\rm Heegner}$ via the  following commutative diagram:
\[
\begin{tikzcd}
\mathrm{Pic}_\mathbb{Q}(\overline{\Sh}(M))^{\mathrm{Heegner}} \arrow[r] \arrow[d, hook] & \mathrm{Pic}_\mathbb{Q}(\Sh(M))^{\mathrm{Heegner}} \arrow[d, "="] \\
\mathrm{Cl}_{\mathbb{Q}}(\overline{\Sh}(M))  \arrow[r, "\sim"'] &  \mathrm{Cl}_{\mathbb{Q}}(\Sh(M))
\end{tikzcd}
\]
Here, the second horizontal map is an isomorphism because   the boundary components of $\Sh(M)$ in $\overline{\Sh}(M)$ have codimension $>1$.

For any element $\cH \in \Pic_\QQ(\overline{\Sh}(M))^{\rm Heegner}$, write
$$\cH =\sum\limits_{\gamma\in G_M} \sum\limits_{m\in \QQ}  a_{m,\gamma}\rH_{m,\gamma}.$$ 
As $\cH$ is $\QQ$-Cartier, its image in each local Picard group is a torsion element.  In particular, for an admissible isotropic plane $J$, $\cH$ is trivial at generic points of \( \partial_J \subseteq  \overline{\Sh}(M) \). Hence by Theorem \ref{thm-BF01}, the pairing of $\cH$ and $\uparrow_{H_J}(\Theta_{L,F})$ vanishes, when $F$ runs over degree 2 harmonic polynomials.

Now suppose an admissible isotropic plane $J$ arises from a decomposition \[M = U(N_1) \oplus U(N_2) \oplus L,\]then $G_M$ has a decomposition of the form 
\[G_M= G_{U(N_1)\oplus U(N_2)}\oplus G_L.\] Fix the isomorphism between $G_L$ and $H_J^\perp/H_J$ as \eqref{decom-iso}. For each $\sigma \in \rO(H_J^\bot/H_J)$, it can be transferred to an element in  $\rO(G_L)$ and can be extended diagonally by identity on $G_{U(N_1)\oplus U(N_2)}$ to an element $\sigma'\in\rO(G_{M})$. By \cite[Thm. 1.14.2]{Ni79}, the natural map $$\rO(M)\rightarrow \rO(G_{M})$$ is surjective, so $\sigma'$ can be further lifted to an element $\tilde{\sigma}\in \rO(M)$.  The same argument applied to $\tilde{\sigma}(J)$ imples that the pairing of $\cH$ and $\uparrow_{H_J}(\sigma^*\Theta_{L,F})$ vanishes, when $F$ runs over degree 2 harmonic polynomials.

In summary, the pairing of $\cH$ with any element in $\mathrm{Cusp}_{\frac{n+2}{2}}^{\mathrm{obs}}(\rho_M^\ast)$ vanishes. Now the equality \( \mathrm{Cusp}_{\frac{n+2}{2}}^{\mathrm{obs}}(\rho_M^\ast) = \mathrm{Cusp}_{\frac{n+2}{2}}(\rho_M^\ast) \) combined with Theorem \ref{thm-heegner=hodge} forces \( \cH \) proportional to the Hodge bundle. 
\end{proof}

\section{Picard group of certain Baily--Borel compactifications}\label{sec6}

In this section, we analyze the negative definite lattices arising in the Baily--Borel compactifications. By applying obstruction theory of extending Heegner divisors, we establish  our main theorem. Additionally, we extend the work of \cite{DFV24} to derive a torsion freeness result for orthogonal Shimura varieties of K3 type.

\subsection{Lattices associated to  Baily--Borel compactifications}
For any isotropic plane \( J \subset M \), one can associate a negative definite  lattice \( J^\perp/J \)  of rank \( n-2 \). A natural question arises:  

\vspace{0.5em}  
\noindent  
\textbf{Question.}  
Is every genus representative of \( G_{J^\perp/J} \)-type negative definite lattice  associated to some isotropic planes of \( M \)? More generally, for any isotropic subgroup \( H \subset G_M \), is there an inclusion  
\begin{equation}\label{genus-inc}  
    \Gen_{0,n-2}(H^\perp/H) \subseteq \left\{ J^\perp/J \mid J \subset M \text{ isotropic plane} \right\}/{\cong} \,?  
\end{equation}  
\vspace{0.5em}  
We establish this under hyperbolic splitting conditions.  
\begin{proposition}\label{prop-genus} 
Suppose $M\cong U(N_1)\oplus U(N_2)\oplus L$  for some $N_i\in \ZZ$ and a negative definite lattice $L$. If $\gcd(N_1,N_2)\mid 2$,  we have 
\begin{equation*}
    \mathbf{Gen}_{0,n-2}(G_L)\subseteq \{ J^\perp/J~\mid~J\subseteq M ~\hbox{is an isotropic plane}\}/\cong.
\end{equation*}
In particular, when $N_1=N_2=1$, we have 
 \begin{equation*}
     \mathbf{Gen}_{0,n-2}(G_M)\subseteq \{ J^\perp/J~\mid~J\subseteq M ~\hbox{is an isotropic plane}\}/\cong.
 \end{equation*}
\end{proposition}
\begin{proof}
Let $L'\in \Gen_{0,n-2}(G_L)$ be a negative definite lattice. Then,
$$L'\oplus U(N_1)\oplus U(N_2) \in \Gen_{2,n}(G_M).$$
 
By Nikulin's uniqueness theorem \cite[Thm. 1.13.1]{Ni79}, the genus \( \Gen_{2,n}(G_M) \) contains exactly one isomorphism class. Thus, we have  
   $$ L'\oplus U(N_1)\oplus U(N_2)\cong M.$$ Let $J$ be the image of the standard isotropic plane in $U(N_1)\oplus U(N_2)$ under this isomorphism. Then $J^{\perp}/J \cong L'$, proving the inclusion.    
\end{proof}

\begin{remark}
    From the proof, it is clear that the right-hand side of the inclusion \eqref{genus-inc} may be restricted to range only over \textit{admissible} isotropic planes (see Subsection \ref{sec-obs}).
\end{remark}

\subsection{Picard Group of $\overline{\Sh}_\Gamma(M)$}

We begin by considering lattices that admit sufficient splitting properties. These lattices serve as a natural setting for applying our arithmetic results.
\begin{definition}
    An even lattice \( M \) is \textit{of K3 type} if it satisfies:
    \begin{enumerate}
        \item \( M \) splits two hyperbolic planes over \( \ZZ \);
        \item \( M \otimes \ZZ_p \) splits three hyperbolic planes for all primes \( p \).
    \end{enumerate}
\end{definition}

Then we have 
\begin{theorem}\label{thm:mainthm2}
Let $M$ be an even lattice of signature $(2,n)$  with $n>8$ and $\Gamma\leq \widetilde{\rO}(M)$. Assume that $M$ is K3 type or $p$-elementary for some prime $p$ and splits one hyperbolic plane. Then $$\dim_\QQ \Pic_\QQ(\overline{\Sh}_\Gamma(M))^{\rm Heegner}=1.$$
\end{theorem}

\begin{proof}
Since we assume that $M$ splits a hyperbolic plane, by Proposition \ref{prop:obs=pic1}, it suffices to show $$\Cusp_k^{\mathrm{obs}}(\rho_M^*) = \Cusp_k(\rho_M^*),$$ for $k=\frac{2+n}{2}$. 

\subsection*{Case 1: $M$ is K3 type.}  

By Proposition \ref{prop-genus}, for any lattice $L'\in \Gen_{0,n-2}(G_M)$, there exists an isotropic plane $J\subset M$ such that $J^\perp/J \cong L'$.  Since $\rho_{L'} \cong \rho_{M}$, applying Theorem \ref{thm-cusp=theta} to $L'$ yields
    \[
    \Cusp^\theta_{\frac{2+n}{2}}(\rho_M^*) = \Cusp^{\mathrm{obs}}_{\frac{2+n}{2}}(\rho_M^*) 
    = \Cusp_{\frac{2+n}{2}}(\rho_M^*).
    \]

\subsection*{Case 2: $M$ is $2$-elementary.}  
Let $G_M \cong (\ZZ/2\ZZ)^a$ with $a \leq n$. Due to the classification of $2$-elementary lattices, $M$ is of K3 type when $a \leq n-4$. Thus we assume $a > n-4$, implying $a \in \{n-2, n\}$. In this situation, by \cite[Thm. 1.10.1 \& 1.14.2]{Ni79}, there exists $N \in \{1, 2\}$ such that $M$ decomposes as:
\[
M = U(N) \oplus U(2) \oplus L
\]
where $L$ is negative definite of rank $n-2$ with $G_L \cong (\ZZ/2\ZZ)^{n-4}$. By \cite[Cor. 1.9.3]{Ni79}, $L$ splits a hyperbolic plane locally, satisfying the assumptions of Theorem \ref{thm-cusp=theta}, hence:
\[
\Cusp_k^\theta(\rho_L^*) = \Cusp_k(\rho_L^*).
\]

For any $L' \in \Gen_{0,n-2}(G_L)$, Proposition \ref{prop-genus} yields an admissible isotropic plane $J \subseteq M$ with $J^\perp/J \cong L'$. Moreover, there exists a natural correspondence between:
\begin{itemize}
    \item Admissible isotropic planes $J \subset M$
    \item Non-characteristic isotropic subgroups $H \leq G_M$ (i.e., containing no nonzero characteristic elements) satisfying $H^\perp/H \cong G_L$
\end{itemize}
This consideration for characteristic elements is necessary due to the failure of Witt's extension theorem in characteristic $2$. Specifically, each admissible $J$ determines a non-characteristic isotropic subgroup $H_J \leq G_M$ with $H_J^\perp/H_J \cong G_L$. Conversely, let $H_{J_0}$ denote the isotropic subgroup associated to the standard isotropic plane $$J_0 \subseteq U(N) \oplus U(2).$$ By \cite[Thm. 3.6.3 \& Lem 3.9.1]{Ni79}, any non-characteristic isotropic subgroup $H \leq G_M$ with $H^\perp/H \cong G_L$ is conjugate to $H_{J_0}$ under some $\rO(M)$-transformation (See also Remark \ref{rmk:conjugate}). This proves $H$ is associated with an admissible isotropic plane.

This correspondence implies:
\[
\uparrow_H \bigl( \Cusp_k(\rho_{H^\perp/H}^*) \bigr) = \uparrow_H \bigl( \Cusp^\theta_k(\rho_{H^\perp/H}^*) \bigr) \subseteq \Cusp^{\mathrm{obs}}_k(\rho_M^*)
\]
for all such non-characteristic isotropic subgroups $H \leq G_M$ and it remains to show that the spaces \(\uparrow_H \bigl( \Cusp_k(\rho_{H^\perp/H}^*) \bigr)\) span \(\Cusp_{k}(\rho_M^*)\).

By \cite[Thm. 5.1]{Ma24a}, any \(f\in\Cusp_{k}(\rho_M^*)\) can be written as 
\begin{equation}\label{eq:isotropicallyLiftedForms}
    f = \sum_H\uparrow_H(f_H)
\end{equation}
where the sum ranges over all isotropic subgroups \(H\leq G_M\) such that \(H^\bot/H\cong(\ZZ/2\ZZ)^b\) with
\[ b \leq \begin{cases}
    6 & \text{if $n$ is even} \\
    5 & \text{if $n$ is odd}
\end{cases} \]
and \(H^\bot/H\) has coparity $1$ if $b=6$ with suitable $f_H\in\Cusp_{k}(\rho_{H^\bot/H}^*)$. \\
{\bf Subcase (i): $a = n-2$.} Since $a=n-2>6$, for any $H$ appearing in \eqref{eq:isotropicallyLiftedForms} we have $|H|>1$. If $|H|>2$, then $H$ contains a subgroup $\tilde{H}\leq H$ of order $2$ which is non-characteristic. Now assume $|H|=2$. When $a$ is odd, $H$ must be non-characteristic. When $a$ is even, since $a\geq8$, it follows that $|H^\bot/H| = 2^6$ and therefore, $H^\bot/H$ has coparity $1$. Hence, $H$ is non-characteristic. In summary, we find that any $H$ appearing in \eqref{eq:isotropicallyLiftedForms} contains a subgroup $\tilde{H}\leq H$ with $|\tilde{H}|=2$ which is non-characteristic. Thus, by the transitivity of the isotropic lifts
\[ f = \sum_H\uparrow_{\tilde{H}}(f_{\tilde{H}}) \]
where $f_{\tilde{H}} = \uparrow_{H/\tilde{H}}(f_H)\in\Cusp_k(\rho_{\tilde{H}^\bot/\tilde{H}}^*)$, which proves the claim. \\
{\bf Subcase (ii): $a = n$.} Similarly to before, since $a=n>8$, for any $H$ appearing in \eqref{eq:isotropicallyLiftedForms} we have $|H|>2$. If $|H|>4$, then $H$ contains a subgroup $\tilde{H}\leq H$ of order $4$ which is non-characteristic. If $|H|=4$, it must be non-characteristic by an analogous argument as before. Let for any $H$ appearing in \eqref{eq:isotropicallyLiftedForms} $\tilde{H}\leq H$ be a subgroup with $|\tilde{H}|=4$ which is non-characteristic. Then,
\[ f = \sum_H\uparrow_{\tilde{H}}(f_{\tilde{H}}) \]
where $f_{\tilde{H}} = \uparrow_{H/\tilde{H}}(f_H)\in\Cusp_k(\rho_{\tilde{H}^\bot/\tilde{H}}^*)$.

\subsection*{Case 3: $M$ is $p$-elementary for $p$ an odd prime.}  
The proof follows a similar strategy to Case 2. The advantage is that Witt's theorem holds, while the difficulty lies in the more intricate classification of $p$-elementary lattices.

Note that $n$ must be even and $n \geq 10$. Let $G_M \cong (\ZZ/p\ZZ)^a$ with $a \leq n$. By \cite[Cor. 1.9.3 and 1.13.5]{Ni79}, $M$ is of K3 type if $a \leq n-5$. Thus, we assume $a \geq n-4$. Again, we claim that there exists $N \in \{1, p\}$ such that $M$ decomposes as:
\[
M = U(N) \oplus U(p) \oplus L
\]
where $L$ is negative definite of rank $n-2$ and splits locally a hyperbolic plane. This is essentially \cite[Cor. 1.10.2 \&  Theorem 1.14.2]{Ni79}.  Specifically, 
\begin{itemize}
    \item If $a \leq n-3$, then $N=1$ and $L$ splits locally a hyperbolic plane by \cite[Cor. 1.9.3]{Ni79}.
    \item If $a > n-3$, then $N=p$. Since $M$ splits globally a hyperbolic plane and $\ell(G_{U(N)\oplus U(p)}) = \rank(U(N)\oplus U(p))$, $L$ splits locally a hyperbolic plane.
\end{itemize}
Here we shall remark that a negative definite $p$-elementary lattice  $L$ does not necessarily locally split a hyperbolic plane when  $\ell(G_L)\leq \mathrm{rank}~L-2$.

Theorem \ref{thm-cusp=theta} then implies:
\[
\Cusp_k^\theta(\rho_L^*) = \Cusp_k(\rho_L^*).
\]
As in Case 2,   one can use Witt's extension theorem to obtain a correspondence between admissible isotropic planes $J \subset M$ and  isotropic subgroups $H \leq G_M$ satisfying $H^\perp/H \cong G_L$. Hence, it suffices to show that the spaces
\[
\uparrow_H \bigl( \Cusp_k(\rho_{H^\perp/H}^*) \bigr) \subseteq \Cusp^{\mathrm{obs}}_k(\rho_M^*),
\]
where $H \leq G_M$ ranges over isotropic subgroups with $H \cong \ZZ/N\ZZ \oplus \ZZ/p\ZZ$, span $\Cusp_k(\rho_M^*)$.

As before, any $f \in \Cusp_k(\rho_M^*)$ decomposes as:
\begin{equation}\label{eq:isotropicallyLiftedFormsP}
    f = \sum_{H} \uparrow_H(f_H)
\end{equation}
where the sum is over isotropic subgroups $H \leq G_M$ satisfying $H^\perp/H \cong (\ZZ/p\ZZ)^b$ with
\[
b \leq \begin{cases} 
4 & \text{if $a$ is even} \\
5 & \text{if $a$ is odd}
\end{cases}
\]
and $f_H \in \Cusp_k(\rho_{H^\bot/H}^*)$. \\

{\bf Subcase (i): $a \leq n-3$.} Since $a\geq n-4\geq 6$, it follows that for any $H$ appearing in \eqref{eq:isotropicallyLiftedFormsP} we have $|H|>1$. Therefore, it contains a subgroup $\tilde{H}\leq H$ with $|\tilde{H}| = p$. Then
\[ f = \sum_H\uparrow_{\tilde{H}}(f_{\tilde{H}}) \]
where $f_{\tilde{H}} = \uparrow_{H/\tilde{H}}(f_H)\in\Cusp_k(\rho_{\tilde{H}^\bot/\tilde{H}}^*)$. \\
{\bf Subcase (ii): $a > n-3$.} Since $a> n-3\geq 7$, it follows that for any $H$ appearing in \eqref{eq:isotropicallyLiftedFormsP} we have $|H|>p$. Therefore, it contains a subgroup $\tilde{H}\leq H$ with $|\tilde{H}| = p^2$. Then
\[ f = \sum_H\uparrow_{\tilde{H}}(f_{\tilde{H}}) \]
where $f_{\tilde{H}} = \uparrow_{H/\tilde{H}}(f_H)\in\Cusp_k(\rho_{\tilde{H}^\bot/\tilde{H}}^*)$.

\end{proof}

\begin{remark}\label{rmk:conjugate}
For $2$-elementary lattices $M$, Nikulin's results \cite[Thm. 3.6.3 \& Lemma 3.9.1]{Ni79}  establish that two non-characteristic elements $\gamma,\gamma'\in G_M$ are $\rO(M)$-conjugate if and only if $\frq_M(\gamma)=\frq_M(\gamma')$. Consequently:
\begin{itemize}
    \item Any two non-characteristic isotropic subgroups of order $2$ are $\rO(M)$-conjugate.
\end{itemize}

Now consider $M = U(2) \oplus U(2) \oplus L$ with a non-characteristic isotropic subgroup $H \leq G_M$ of order $4$. Let $\gamma_1, \gamma_2$ generate $H$ and $e_i,f_i$ the standard basis of $U(2)\oplus U(2)$. By Nikulin's result:
\begin{itemize}
    \item We may assume $\gamma_1 = e_1^*$ after $\rO(M)$-action;
    \item $\gamma_2 \in G_{U(2) \oplus L}$;
\end{itemize}
Applying $\rO(U(2) \oplus L)$ then maps $\gamma_2$ to the generator of the second $U(2)$ factor. 
\end{remark}

   \begin{remark}
For a $p$-elementary lattice $M$ of signature $(2, n)$, there is a numerical criterion for $M$ to split a hyperbolic plane:
\begin{itemize}
    \item If $p = 2$, then $M$ splits a hyperbolic plane if and only if $\ell(G_M) \leq n$.
    \item If $p > 2$, then $M$ splits a hyperbolic plane if and only if $\ell(G_M)<n$ or $\ell(G_M)=n$ and $n \equiv 2 \pmod{8}$.
\end{itemize}
This is due to the work of Nikulin \cite{Ni79}.
\end{remark}

\begin{corollary}\label{cor:mainthm2}
Let $M$ be an even lattice of signature $(2,n)$ with $n>8$. We have
$$\dim_\QQ \Pic_\QQ(\overline{\Sh}_\Gamma(M))=1$$
if either $M$ is of K3 type and $\widetilde{\rO}(M)\leq \Gamma$ or $M$ is $p$-elementary and  $\Gamma=\Gamma_0$. 
\end{corollary}
\begin{proof}
  As $\dim M\geq 3$,  the natural projection $\Sh(M)\to \Sh_\Gamma(M)$ preserves the Hodge line bundle and extends to a finite morphism 
    \[\overline{\pi}: \overline{\Sh}(M)\rightarrow \overline{\Sh}_\Gamma(M),\]
between their Baily--Borel compactifications. 
This induces a commutative diagram
   \[
\begin{tikzcd}[column sep=large]
\Pic_{\mathbb{Q}} \left( \overline{\Sh}_{\Gamma}(M) \right) \ar[r, "\overline{\pi}^*"] \ar[d, hookrightarrow] & \Pic_{\mathbb{Q}} \left( \overline{\Sh}(M) \right) \ar[d,hookrightarrow]\\
\Pic_{\mathbb{Q}} \left(\Sh_\Gamma(M) \right) \ar[r, "\pi^\ast "] & \Pic_{\mathbb{Q}} \left(\Sh(M) \right)
\end{tikzcd}
\]

According to Corollary \ref{cor:Heeg-span}, in either case, this map will factor through the Heegner divisor subspaces $\Pic_\QQ(\overline{\Sh}(M))^{\rm Heegner}$.  Thus we can conclude the assertion by Theorem \ref{thm:mainthm2}.
\end{proof}

\subsection*{Proof of Corollary \ref{cor:pic-cy}} Under our hypothesis, the classification result in \cite{AE22} shows that  $\Lambda_\rho$ is a $2$-elementary lattice and  splits a hyperbolic plane. It follows from Corollary \ref{cor:mainthm2} and Remark \ref{rem:Spinorkernel} that $\Pic ((\Gamma_\rho\backslash \bD_\rho)^{\rm BB})$ has Picard number one. 

\subsection{Torsion freeness of the Picard group}\label{tor-free}
We first generalize the results of \cite{DFV24} to Shimura varieties associated with K3-type lattices of large rank.

\begin{definition}
Let $[\Sh(M)]=[\widetilde{\rO}(M)\backslash \bD]$ be the (analytic) quotient stack. It is well-known that $[\Sh(M)] $ is a smooth Deligne--Mumford stack and its coarse moduli space is $\Sh(M)$. 
  
We can define an analytic line bundle on $[\Sh(M)]$ as follows: take the trivial line bundle 
$$\bD \times \mathbb{A}^1$$ 
over $\bD$ and let $\Gamma=\widetilde{\rO}(M)$ act diagonally, where the action on $\mathbb{A}^1$ is given by $$A \cdot \lambda \coloneqq  \det(A)\lambda.$$ The resulting quotient $\cL_M \coloneqq  [\Gamma \backslash \bD \times \mathbb{A}^1]$ is a line bundle over $[\Sh(M)]$.
\end{definition}

Note that $\cL_M$ is not trivial because the determinant of an element in $\Gamma$ is not trivial in general.

For any prime $p$, denote by
$r_p(M)$ the maximal rank of the sublattices $L$ in $M$ such that $\det(L)$ is coprime to $p$. Then we have the following general result.
\begin{theorem}\label{thm:torsion-free}
 Let $M$ be an even lattice of signature $(2,n)$ which splits two hyperbolic lattices.  Assume that $r_2(M)\geq 6$ and $r_3(M)\geq 5$. Then  $\Pic(\Sh(M))$  is torsion free and the torsion part of $\Pic([\Sh(M)])$ is isomorphic to $\ZZ/2\ZZ$, which is generated by $\cL_M$. 
\end{theorem}
\begin{proof}
By \cite[Thm. 1.7]{GHS09}, the abelianisation of \(\widetilde{\mathrm{O}}(M)\) is isomorphic to \(\mathbb{Z}/2\mathbb{Z}\). Following the argument in the proof of \cite[Proposition 2.4]{DFV24}, the torsion part of \(\mathrm{Pic}([\mathrm{Sh}(M)])\) is given by
\[
    \mathrm{Pic}([\mathrm{Sh}(M)])_{\mathrm{torsion}} = \mathbb{Z}\langle \mathcal{L}_M \rangle \cong \mathbb{Z}/2\mathbb{Z}.
\]
This completes the proof of the last assertion.

To show that \(\mathrm{Pic}(\mathrm{Sh}(M))\) is torsion-free, we note that by \cite[Proposition 6.1]{Ols12}, there is an injection
\[
    \mathrm{Pic}(\mathrm{Sh}(M)) \hookrightarrow \mathrm{Pic}([\mathrm{Sh}(M)]).
\]
via the pullback of $[\Sh(M)]\to \Sh(M)$.
Thus, it suffices to prove that \(\mathcal{L}_M\) does not descend to \(\mathrm{Sh}(M)\). Set $M_0=U^{\oplus 2 }$ and write $M=M_0\oplus N$. Then the natural embedding 
$$M_0\hookrightarrow M $$
induces a Cartesian diagram
\begin{equation*}
    \begin{tikzcd}
        \left[\mathrm{Sh}(M_0)\right] \ar[r] \ar[d] & \left[\mathrm{Sh}(M) \right] \ar[d] \\
        \mathrm{Sh}(M_0) \ar[r] & \mathrm{Sh}(M).
    \end{tikzcd}
\end{equation*}
The pullback of $\cL_{M}$ to $[\Sh(M_0)]$ is $\cL_{M_0}$. If $\cL_{M}$ descends to $\Sh(M)$, then $\cL_{M_0}$ also descends to $\Sh(M_0)$. However,   this contradicts the fact $\Sh(M_0)\cong \AA^2$ and only the trivial line bundle descends to $\Sh(M_0)$.

\end{proof}

As an immediate consequence, we obtain
\begin{corollary}\label{cor:mainthm3}
Let $M$ be a lattice of K3 type and  $\rank(M)>10$. Then $\Pic(\overline{\Sh}(M))\cong \ZZ$ is spanned by some multiple of $\overline{\lambda}_M$. 
\end{corollary}

\begin{proof}
Since \( M \) is of K3 type, it satisfies the conditions \(r_2(M) \geq 6\) and \(r_3(M) \geq 5\). By Theorem \ref{thm:torsion-free}, \(\Pic(\Sh(M))\) is torsion-free. As \(\overline{\Sh}(M) \setminus \Sh(M)\) has codimension \(> 2\), \(\Pic(\overline{\Sh}(M))\) is also torsion-free. This implies \(\Pic(\overline{\Sh}(M)) \cong \mathbb{Z}\).
\end{proof}

\section*{Appendix}
\appendix
\section{A vector-valued Siegel--Weil formula}\label{appA}
Following \cite{Ma24} and \cite{MpSW}, we introduce the Siegel--Weil formula for metaplectic groups. 
For simplicity, let us assume $M$ is a positive definite even lattice  and $V=M\otimes \QQ$.  Let $\Sp_{2d}$ be the standard symplectic $\ZZ$-group scheme and $\Mp_{2d}(\AA)\stackrel{\pi}{\rightarrow}\Sp_{2d}(\AA)$ be the metaplectic double cover. Then the inclusion $\Sp_{2d}(\QQ)\hookrightarrow \Sp_{2d}(\AA)$ can be uniquely extended to an inclusion $$\Sp_{2d}(\QQ)\hookrightarrow \Mp_{2d}(\AA).$$ 
Through this lifting, we will consider $\Sp_{2d}(\QQ)$ as a subgroup of $\Mp_{2d}(\AA)$.

Fix the standard additive character $\psi:\QQ\backslash\AA \rightarrow \CC^{\times}$ whose archimedean component is given by $\psi_{\infty}:\RR \rightarrow \CC^{\times}$, $x_{\infty}\mapsto \bfe(x_{\infty})$ and the $p$-adic component is given by $\psi_p:\QQ_p \rightarrow \CC^{\times}$, $x_p\mapsto \bfe(-x_p')$, where $x_p'\in \QQ/\ZZ$ is the principal part of $x_p$. Let $\omega_\psi$  be the (automorphic) Weil representation of $\rO(V)(\AA)\times \Mp_{2d}(\AA)$ realized in the Schr$\ddot{\hbox{o}}$dinger model $\cS(V(\AA)^d)$,  where $\cS(V(\AA)^d)$ is the space of Schwartz--Bruhat functions on $V(\AA)^d$.  Similarly, one can define $\omega_{\psi,f}$ as the Weil representation acting on $\cS(V(\AA_f)^d)$. 

\begin{remark}\label{ref:weil-rep comparison}
    There is a natural relation between this automorphic Weil representation and the representation $\rho_M^{(d)}$ defined in Subsection \ref{subsec:weil-rep}. Each element $\gamma\in G_M^{(d)}$ corresponds to a Schwartz function $$\varphi_{\gamma}=\otimes_{p<\infty}\varphi_p \in \cS(V(\AA_f)^d),$$ where $\varphi_p \in \cS(V(\QQ_p)^d)$ is the characteristic function of $\gamma+(M\otimes \ZZ_p)^d$. Under the map
\begin{align*}
    \iota:\CC[G_M^{(d)}] & \rightarrow \cS(V(\AA_f)^d)\\
    \fre_{\gamma} &\mapsto \varphi_{\gamma}, 
\end{align*}
we have $\omega_{\psi,f}(g_f)\circ \iota=\iota\circ \overline{\rho_{M}^{(d)}}(g)$ (cf.~\cite{Zhang2009ModularityOG}) for $g\in\Mp_{2d}(\ZZ)$ and $g_f\in\Mp_{2d}(\hat{\ZZ})$ the unique element such that $gg_f\in\Sp_{2d}(\QQ)\subset\Mp_{2d}(\AA)$.

\end{remark}

\begin{definition}
   For $\varphi\in\cS(V(\AA)^d)$, we can define the theta series 
    \begin{equation*}
        \theta(g,h,\varphi)\coloneqq \sum_{\bfx\in V(\QQ)^d}\omega_{\psi}(g)\varphi(h^{-1}\bfx), \quad g \in \Mp_{2d}(\AA), \quad h \in \rO(V)(\AA). 
    \end{equation*}
   which is automorphic on both $\Mp_{2d}(\AA)$ and $\rO(V)(\AA)$.

\end{definition}

Let $P=NA\subset \Sp_{2d}$ be the standard Siegel parabolic subgroup. Here for any commutative ring $R$,
$$A(R)\coloneqq \left\{m(U)\coloneqq  \left(\begin{array}{cc}U & 0 \\0 & (U^{-1})^\top\end{array}\right) | U \in \GL_d(R) \right\},$$
and
$$N(R)\coloneqq \left\{n(B)\coloneqq  \left(\begin{array}{cc}I_d & B \\0 & I_d \end{array}\right) | B \in \sym_d(R) \right\}.$$
We have the global Iwasawa decomposition $\Sp_{2d}(\AA)=N(\AA)A(\AA)K$ and $\Mp_{2d}(\AA)=N(\AA)\widetilde{A}(\AA)\widetilde{K}$ for the standard maximal open compact subgroup. 

\begin{definition}
   With the notations as above,  the Eisenstein series associated with $\varphi\in \cS(V^d(\AA))$ is defined by 
 \begin{equation*}
     E(g,s,\varphi)\coloneqq \sum\limits_{\gamma\in P(\QQ)\backslash \Sp_{2d}(\QQ)} \Phi(\gamma g,s,\varphi) \quad g \in \Mp_{2d}(\AA), \quad s\in\CC,
 \end{equation*}
 where $\Phi(g,s,\varphi)=|\det a(g)|^{s-s_0}(\omega_{\psi}(g)\varphi)(0)$ with $s_0=\frac{r}{2}-\frac{d+1}{2}$ and $\pi(g)=n\cdot m(a(g))\cdot k\in\Sp_{2d}(\AA)$ under the Iwasawa decomposition.   It converges absolutely for $\mathrm{Re}(s)>\frac{d+1}{2}$.
\end{definition}

The famous Siegel--Weil formula identifies the theta lifting of the constant function $\mathbf{1}$ and the special value of the Eisenstein series.

\begin{theorem}[Siegel--Weil formula in the metaplectic case (cf. \cite{MpSW})]\label{thm:MpSW}
       Let $\varphi\in \cS(V(\AA)^d)$ be a $\widetilde{K}$-finite function. Suppose $r>d+3$, then the Eisenstein series $E(g,s,\varphi)$ is holomorphic at $s=s_0$ and
        \begin{equation*}
        E(g,s_0,\varphi)=\int_{\mathrm{O}(V)(\QQ)\backslash \mathrm{O}(V)(\AA)}\theta(g,h,\varphi)\rmd h.
        \end{equation*}
        Here the Haar measure $\rmd h$ is normalized so that $\vol(\mathrm{O}(V)(\QQ)\backslash \mathrm{O}(V)(\AA))=1.$
    \end{theorem}

As an application, we can deduce Theorem \ref{thm:Eis=Theta}.

\begin{theorem}\label{thm:vecSW}
    Suppose $M$ is a positive definite even lattice of rank $r$ with $\frac{r}{2}>d+1$. Then \[\Theta^{(d)}_{\mathbf{Gen}(M)}=\bE_{\frac{r}{2},M}^{(d)}.\]
\end{theorem}
\begin{proof}
 The idea is to identify the component functions of both sides via the Siegel--Weil formula. When $r$ is even, this is \cite[Thm. 5.5]{Ma24}. Now suppose $r$ is odd and we sketch the proof as below.

According to Kudla (cf. \cite[p. 37 Proposition 4.3]{Kudla96}), for the archimedean component of the Weil representation $\omega_{\psi}$, there is a character $\chi_{V,\infty}^{\psi}$ of $\widetilde{A(\RR)}$ such that
    \[(\omega_{\psi,\infty
    }(\tilde{m})\varphi_{\infty})(\bfx)=\chi_{V,\infty}^{\psi}(\tilde{m})|\det(m(a))|^{\frac{r}{2}}\varphi_\infty(\bfx\cdot a),\]
    where $\tilde{m}=(m(a),\phi_{m(a)})\in \widetilde{A(\RR)}$, $a\in\GL_d(\RR)$ and $\varphi_\infty\in \cS(V(\RR)^d)$. 
    
    Let $\varphi_{\infty}(\bfx)=e^{\pi \tr(\bfx,\bfx)}$ be the standard Gaussian function and for $\gamma \in (M^{\vee}/M)^d$, denote $\varphi_{\infty}\otimes \varphi_{\gamma}$ by $\varphi_{\infty,\gamma}$. For $\tau=x+iy\in\HH_d$ and $a\in\GL_d(\RR)$ satisfying $aa^\top=y$, we consider the element $$g_{\tau}=n(x)m(a)\in \Sp_{2d}(\RR).$$ The component function of our previously defined vector-valued genus theta series (resp.\ vector-valued Siegel--Eisenstein series) is given as follows.
    \begin{itemize}[leftmargin=2em]
        \item \textbf{Genus theta series:}
        \begin{multline*}
            \langle\Theta^{(d)}_{\mathbf{Gen}(M)}(\tau),\fre_\gamma\rangle \\
            =\chi_{V,\infty}^{\psi}(\tilde{m})^{-1}|\det(m(a))|^{-\frac{r}{2}}\int_{\mathrm{O}(V)(\QQ)\backslash \mathrm{O}(V)(\AA)}\theta(\widetilde{g}_{\tau},h,\varphi_{\infty,\gamma})\rmd h,
        \end{multline*}
        where $\widetilde{g}_{\tau}=n(x)\tilde{m}=(g_\tau,\phi_{g_\tau})$ and the Haar measure $\rmd h$ is normalized so that $$\vol(\mathrm{O}(V)(\QQ)\backslash \mathrm{O}(V)(\AA))=1.$$ 
        \item \textbf{Siegel--Eisenstein series:}
        \[\langle\bE_{\frac{r}{2},M}^{(d)}(\tau),\fre_\gamma\rangle=\chi_{V,\infty}^{\psi}(\tilde{m})^{-1}|\det(m(a))|^{-\frac{r}{2}}E(\widetilde{g}_{\tau},s_0,\varphi_{\infty,\gamma}).\]
        
    \end{itemize}
    For the comparison of theta series, the proof is exactly the same as in the classical case (cf. \cite[Example 2.2.6]{LiChao21}). For more details, one can see \cite[Thm. 5.5]{Ma24}. For the comparison of Siegel--Eisenstein series, since \cite[Thm. 5.5]{Ma24} uses a computation by Kudla which might be difficult to locate a reference in the metaplectic case, we briefly present an alternative proof here.

    Consider $q_{\tau}(\bfx)\coloneqq e^{-\pi i\tr((\bfx,\bfx)\tau)}$. One has
    \[\omega_{\psi,\infty}(n(B))(q_{\tau}(\bfx))=\psi_{\infty}\left(\frac{1}{2}\tr((\bfx,\bfx)B)\right)q_{\tau}(\bfx)=q_{n(B)\cdot\tau}(\bfx)\]
    for $n(B)=\Big(\MU, 1\Big)$, where $B\in \sym_d(\ZZ)$ and 
    \[\omega_{\psi,\infty}(J_d)(q_{\tau}(\bfx))=\sqrt{\det(\tau)}^{-r}q_{J_d\cdot\tau}(\bfx)\] for $J_d=\Big(\MJ, \sqrt{\det(\tau)}\Big)$. Since $\Mp_{2d}(\ZZ)$ is generated by $J_d$ and $n(B)$'s, combining these equations one gets
    \[\omega_{\psi,\infty}(\tilde{g})(q_{\tau}(\bfx))=\phi_g(\tau)^{-r}q_{g\cdot\tau}(\bfx)\]
    for $\tilde{g}=(g,\phi_g(\tau))\in\Mp_{2d}(\ZZ)$.\footnote{According to Shimura \cite{Shi85}, this is basically how one gets the classical definition of $\Mp_{2d}(\RR)$ as in Section \ref{sec2} via the Weil representation. For details, one can see for example \cite[Proposition 2.4, Corollary 2.5]{Hida07}.}
    Hence \[\chi_{V,\infty}^{\psi}(\tilde{m})^{-1}|\det(m(a))|^{-\frac{r}{2}}\Phi(\tilde{g}\cdot\widetilde{g}_\tau,s_0,\varphi_{\infty,\gamma})=\phi_g(\tau)^{-r}\omega_{\psi,f}(\tg_f)(\varphi_\gamma)(0)\]
    and the rest are exactly the same as the last several lines in \cite[Thm. 5.5]{Ma24}.

The result now follows from Theorem \ref{thm:MpSW}. Note that in our case $r>2d+2\geq d+3$ for all $d\in \ZZ_{>0}$. So the condition $r>d+3$ in Theorem \ref{thm:MpSW} always holds.

\end{proof}

\section{Coset decomposition and Weil representations}\label{appB}

Recall there is a map
\begin{equation*}
    \iota: \Mp_2(\ZZ)\times \Mp_2(\ZZ)\to \Mp_4(\ZZ). 
\end{equation*}
For $\tA=(A,\phi_A),\tB=(B,\phi_B)\in \Mp_2(\ZZ)$,  we define $u(\tA)$ to be the image of $\iota(\tA, (\rI_2, 1))$ and $d(\tB)$ to be the image of $\iota((\rI_2,1), \tB)$. If we set \begin{center}
     $u(\tA)=(u(A), \phi_{u(A)})$ and $d(\tB)=(d(B),\phi_{d(B)})$,
 \end{center}
 with $u(A),d(B)\in \SL_2(\ZZ)$,   then for $(z,z')\in \HH\times \HH \subseteq \HH_2$, one has 
 \begin{equation*}
      \phi_{u(A)}(z,z')=\phi_A(z)~\hbox{and}~ \phi_{d(B)}(z,z')=\phi_B(z')
 \end{equation*}
 from the definition.  Next, for any $\alpha\in \ZZ$, we define  $$\widetilde{\cC}_\alpha=(\cC_\alpha,\phi_{\cC_\alpha}(\tau))=\Big(\footnotesize{\begin{pmatrix}
      \alpha^2+\alpha & -\alpha-1& -1& -\alpha-1\\
      -\alpha-1& 1& 0 & 0\\
      -\alpha & 1 & 0 &0 \\
        0& 0 & -1& -\alpha
  \end{pmatrix}}, \sqrt{\alpha^2z-2\alpha w+z' } ~\bigg) ,$$
 with $\tau=\begin{pmatrix}
      z &w\\ w& z' \end{pmatrix}\in \HH_2$.

It has been shown in \cite[Proposition 4.2]{Ma24} that for any $g^{(2)}\in\Sp_4(\ZZ)$ there exist some $\alpha\in \ZZ^{\leq 0}, A, B\in \SL_2(\ZZ)$ such that 
\[ \Gamma_\infty^{(2)}g^{(2)} = \Gamma_\infty^{(2)}\cC_\alpha u(A)d(B) \]
and the map
\begin{align*}
    \varphi:\Sp_4(\ZZ)&\to\mathrm{Mat}_2(\ZZ), \\
    \varphi(D\cC_\alpha u(A)d(B))&= \pm B'\begin{pmatrix}
       \alpha^2 &0\\0&1
   \end{pmatrix} A,
\end{align*}
\[   \]
is well-defined, where $D=n(S)m(U)\in\Gamma_\infty^{(2)}$ with $\det(U)=\pm1$ and $B'=\begin{pmatrix}
    d &b\\c& a
\end{pmatrix}$ if $B=\begin{pmatrix}
    a &b\\ c& d
\end{pmatrix}$. Since $(I,-1)\in\widetilde{\Gamma}_\infty^{(2)}$, it follows that we also get a decomposition
\begin{equation}\label{eq:B.decomposition}
    \widetilde{\Gamma}_\infty^{(2)}\tg^{(2)} = \widetilde{\Gamma}_\infty^{(2)}\widetilde{\cC}_\alpha u(\tA)d(\tB)
\end{equation}
for any $\tg^{(2)}\in\Mp_4(\ZZ)$. For our purpose,  we need to extend the map $\varphi$ to metaplectic covers and study their actions under the Weil representation.

\begin{definition}\label{defn:bijection}

\begin{itemize}[leftmargin=2em]
    \item For $\tB=(B,\phi_B)\in \Mp_2(\ZZ)$, we define $\tB'=(B',\phi_{B'})$ where $\phi_{B'}$ is given by  $$\phi_{B'}(z)\sqrt{z'+B'z}=\phi_B(z')\sqrt{z+Bz'}.$$
    for $z,z'\in \HH$.
    \item  For $\tg^{(2)}=(g^{(2)},\phi_{g^{(2)}})\in \Mp_4(\ZZ)$,  we define
$  \widetilde{\varphi}(\tg^{(2)})=(\varphi(g^{(2)}),\phi_{\varphi(g^{(2)})})$, where $\phi_{\varphi(g^{(2)})}$ is chosen to satisfy
\begin{equation}\label{eq:kern-eq3}
   \phi_{g^{(2)}}(z,z')= \phi_{\varphi(g^{(2)})}(z)\sqrt{z'+\varphi(g^{(2)})z}.
\end{equation}
\end{itemize}

\end{definition}

Under this map, we have $\widetilde{\varphi}(\widetilde{\cC}_\alpha)=\mathbf{\tg}_\alpha=\Big(\begin{pmatrix}
           \alpha^2 & 0\\0 & 1
       \end{pmatrix},1\Big)$
and one verifies that $\phi_{B'}(z)$ and $\phi_{\varphi(g^{(2)})}(z)$ do not depend on $z'$.
\begin{lemma}
    For any $\tA,\tB\in \Mp_2(\ZZ)$,  we have
   \begin{enumerate}[label={$(\roman*)$}]
       \item $(\tA \tB)'=\tB'\tA'$
       \item $\widetilde{\varphi}(\widetilde{\cC}_\alpha u(\tilde{A})d(\tilde{B}))= \tB' \cdot \widetilde{\bfg}_\alpha\cdot \tA$
   \end{enumerate} 
\end{lemma}

\begin{proof}
For $(i)$,   from the definition, one can directly verify that 
    \begin{align*}
        \phi_{(A \cdot B)'}(z)\sqrt{z'+(A \cdot B)'z} 
        &= \phi_{A'}(z)\phi_{B}(z')\sqrt{A'z+B z'} \\
        &=\phi_{B'\cdot A'}(z)\sqrt{z'+(A\cdot B)'z}.
    \end{align*}
    
For $(ii)$,  set $\widetilde{\cC}_{\alpha}\cdot u(\tilde{A})\cdot d(\tilde{B})=(g^{(2)},\phi_{g^{(2)}})$.  As $\varphi(g^{(2)})=B'\cdot \bfg_\alpha\cdot A$, it suffices to show that $$ 
\frac{\phi_{g^{(2)}}(z,z')}{\sqrt{z'+\varphi(g^{(2)})}z}=\phi_{B'}(\bfg_\alpha\cdot Az) \phi_A(z).$$ This is clear because  we have 
    \begin{equation}
        \begin{aligned}
            \mathrm{LHS}&=\frac{\left(\sqrt{\alpha^2Az+Bz'} \right)\phi_{u(A)}(z,Bz')\phi_{d(B)}(z,z')}{\sqrt{z'+B'\alpha^2Az}}\\
            &=\frac{\left(\sqrt{\alpha^2Az+Bz'} \right)\phi_{A}(z)\phi_{B}(z')}{\sqrt{z'+B'\alpha^2Az}}\\
            &=\phi_{B'}(\alpha^2 Az)\phi_A(z)\\
            &=\mathrm{RHS}.
        \end{aligned}
    \end{equation}
\end{proof}

Now it follows from \cite[Proposition 4.2]{Ma24} that there is a well-defined injection
\begin{equation}
\begin{aligned}
    \widetilde{\Gamma}^{(2)}_\infty\backslash \Mp_4(\ZZ)& \to ({\bf \widetilde{Mat}}_2(\ZZ)\times \ZZ^{\leq 0})/(\langle (-I,i)\rangle\times0) \\  [\tg^{(2)}] &\mapsto (\varphi(\tg^{(2)}), \alpha).
\end{aligned} 
\end{equation}

We get

\begin{lemma}[{\cite[Lem. 4.3]{Ma24}}]\label{lem:mu lem4.3}
    Let $\alpha \in \ZZ$, $\tA \in \widetilde{\bY}_{\alpha^2}$ and $\tB \in \Mp_2(\ZZ)$. Then
    $$\sum_{\gamma \in G_M}(\fre_{\gamma}\mid[\tA])\otimes \rho_M(\tB)^{-1}\fre_{\gamma}=\sum_{\gamma \in G_M}(\fre_{\gamma}\mid [\tB'\tA])\otimes\fre_{\gamma}.$$
\end{lemma}

Then we have 
\begin{proposition}\label{prop:coset-action}For $\tg^{(2)}=\widetilde{\cC}_\alpha u(\tA) d(\tB)$ with $\alpha\leq 0$, we have     \begin{equation}\label{eq:B.3}
 \rho^{(2)}_M(\tg^{(2)})^{-1}(\fre_0\otimes \fre_0)=\frac{\bfe(\sign(G_M)/8)}{\sqrt{|G_M|}}\sum\limits_{\gamma\in G_M} (\fre_\gamma\mid [\widetilde{\varphi}(\tg^{(2)})])\otimes \fre_\gamma.
\end{equation}
\end{proposition}
\begin{proof}
It is straightforward to check that $\widetilde{\cC}_{\alpha}$ admits the following decomposition
$$\widetilde{\cC}_{\alpha}=J_2^{-1}\cdot n\left(\begin{pmatrix}
       0 &-1\\-1& -\alpha
   \end{pmatrix}\right)\cdot J_2^{-1}\cdot n\left(\begin{pmatrix}
       \alpha^2+\alpha &-\alpha-1\\-\alpha-1& 1
   \end{pmatrix}\right)\cdot J_2^{-1}\cdot n\left(\begin{pmatrix}
       0 &0\\0& 1
   \end{pmatrix}\right)$$
 in $\Mp_4(\ZZ)$. 

Given this decomposition, when $\tg^{(2)}=\widetilde{\cC}_{\alpha}$, the equation \ref{eq:B.3} holds by a direct computation (cf.~\cite[Proposition 4.4]{Ma24}).\footnote{A careful reader may note that since our definition of "$J$" is different from that in \cite{Ma24}, the action of "$J_2$" in \cite{Ma24} is actually the action of "$J_2^{-1}$" in our notation. Therefore, the reason why we use $J_2^{-1}$ in the above decomposition is to keep the consistency with the computations in \cite[Proposition 4.4]{Ma24}.} Then from the decomposition \eqref{eq:B.decomposition},  we have 
\begin{align*}
    \rho_M^{(2)}&(\tg^{(2)})^{-1}(\fre_0 \otimes \fre_0) \\
    &= \rho_M^{(2)}(\widetilde{\cC}_{\alpha} u(\tA) d(\tB))^{-1}(\fre_0 \otimes \fre_0) \\
    &= \rho_M^{(2)}(u(\tA) d(\tB))^{-1} \rho_M^{(2)}(\widetilde{\cC}_{\alpha})^{-1}(\fre_0 \otimes \fre_0) \\
    &= \rho_M^{(2)}(u(\tA) d(\tB))^{-1} \frac{\bfe(\sign(G_M)/8)}{\sqrt{|G_M|}}\sum\limits_{\gamma\in G_M} (\fre_\gamma\mid[\varphi(\widetilde{\cC}_{\alpha})])\otimes \fre_\gamma \\
    &= \frac{\bfe(\sign(G_M)/8)}{\sqrt{|G_M|}}\sum\limits_{\gamma\in G_M} (\fre_\gamma\mid[\varphi(\widetilde{\cC}_{\alpha})]\mid[\tA])\otimes \rho_M(\tB)^{-1}(\fre_\gamma) \\
    &= \frac{\bfe(\sign(G_M)/8)}{\sqrt{|G_M|}}\sum\limits_{\gamma\in G_M} (\fre_\gamma\mid[\tB' \varphi(\widetilde{\cC}_{\alpha})\tA])\otimes \fre_\gamma \\
    &= \frac{\bfe(\sign(G_M)/8)}{\sqrt{|G_M|}}\sum\limits_{\gamma\in G_M} (\fre_\gamma\mid[\varphi(\widetilde{\cC}_{\alpha} u(\tA) d(\tB))])\otimes \fre_\gamma \\
    &= \frac{\bfe(\sign(G_M)/8)}{\sqrt{|G_M|}}\sum\limits_{\gamma\in G_M} (\fre_\gamma\mid[\varphi(\tg^{(2)})])\otimes \fre_\gamma.
\end{align*}
Here we use Lemma \ref{lem:mu lem4.3} in the fifth equation.
\end{proof}

\section{Hecke kernels}\label{appHK}

\begin{proposition}\label{prop:Hecke-Kernel}
For $\alpha>0$, we have 
    \begin{equation}
    \frac{\bT_{\alpha^2}(f)(z')}{\alpha^{2k-2}}=C\int_{\SL_2(\ZZ)\backslash \HH}\langle f,  K_\alpha(z,-\overline{z}')\rangle y^{k}\frac{\rmd x \rmd y}{y^2},
\end{equation}
 where the constant $C=\frac{2^{k-4}(k-1)}{\pi i^k}$.
\end{proposition}

\begin{proof}
We only deal with the odd signature case. As in the even signature case,  we have $\bT_{\alpha^2}(K_\alpha(z,z'))=\alpha^{2k-2}K_\alpha(z,z')$ and it suffices to show that
\begin{equation*}
    \bT_1(f)(z')=f(z')=C\int_{\SL_{2}(\ZZ)\backslash \HH} \langle f, K_1(z,-\bar{z}')\rangle y^{k}\frac{\rmd x \rmd y}{y^2}.
\end{equation*}
Denote the element $\left(T=\left(\begin{smallmatrix}
    1 &1\\0 &1
\end{smallmatrix}\right), 1\right)\in\Mp_2(\ZZ)$ by $\widetilde{T}$. Then $K_1(z,z')$ equals

\begin{equation}\label{hecke=poincare}
\begin{aligned}
    &\quad \sum\limits_{\gamma\in G_M} \sum\limits_{\tg\in \Mp_2(\ZZ) } \frac{\phi_g(z)^{-2k}}{(z'+g\cdot z)^k}\rho_M(\tg)^{-1}(\fre_\gamma)\otimes \fre_\gamma\\
   &=\sum\limits_{\gamma\in G_M} \sum\limits_{\tg\in \langle\widetilde{T}\rangle \backslash\Mp_2(\ZZ) } \sum\limits_{n\in \ZZ} \frac{\phi_{T^ng}(z)^{-2k}}{(z'+T^n g\cdot z)^k}\rho_M(\widetilde{T}^n\tg)^{-1}(\fre_\gamma)\otimes \fre_\gamma\\
   &=\sum\limits_{\gamma\in G_M} \sum\limits_{\tg\in \langle\widetilde{T}\rangle \backslash\Mp_2(\ZZ) } \sum\limits_{n\in \ZZ} \frac{\phi_{g}(z)^{-2k}}{(z'+g\cdot z+n)^k} \bfe(-n\frq(\gamma))\rho_M(\tg)^{-1}(\fre_\gamma)\otimes \fre_\gamma\\
   &= 2\frac{(-2\pi i)^k}{\Gamma(k)} \sum\limits_{\gamma\in G_M}\sum\limits_{\substack{m\in\ZZ+\frq(\gamma)\\m>0}} m^{k-1}\bP_{\gamma, m}(z)\otimes \bfe(mz')\fre_\gamma
\end{aligned}
\end{equation}
where $\bP_{\gamma,m}(z)=\frac{1}{2}\sum\limits_{\tg \in \widetilde{\Gamma}_\infty^{+}\backslash\Mp_2(\ZZ)} (\bfe(mz)\fre_\gamma) \mid_{k,G_M}[\tg] $ is the Poincar\'e series defined in \cite[\S 1.2]{Br02}. Here, the last identity comes from the Lipschitz Summation Formula in \cite{Lip89} (see also \cite[Thm. 1]{PPA01}), which we recall in Lemma \ref{lem:summation} below.

Write the cusp form $f$ as a Fourier expansion
\[f(z)=\sum\limits_{\gamma\in G_M}\sum\limits_{\substack{m\in\ZZ+\frq(\gamma)\\m>0}} c(m,\gamma)\bfe(mz)\fre_\gamma.\]
By \cite[Proposition 1.5]{Br02}, one has $\langle f,\bP_{\gamma,m}(z)\rangle_{\rm Pet}=\frac{2\cdot\Gamma(k-1)}{(4\pi m)^{k-1}}c(m,\gamma)$. Combine this with \eqref{hecke=poincare}, we finally get

\begin{align*}
&\int_{\SL_{2}(\ZZ)\backslash \HH} \langle f, K_1(z,-\bar{z}')\rangle y^{k}\frac{\rmd x \rmd y}{y^2}\\
=
&2\frac{(2\pi i)^k}{\Gamma(k)}\frac{2\cdot\Gamma(k-1)}{(4\pi)^{k-1}}\sum\limits_{\gamma\in G_M}\sum\limits_{\substack{m\in\ZZ+\frq(\gamma)\\m>0}} c(m,\gamma)\bfe(mz')\fre_\gamma
\\=&\frac{\pi i^k}{2^{k-4}(k-1)}f(z').&
\end{align*}
\end{proof}

\begin{lemma}[Lipschitz Summation Formula]\label{lem:summation}
Let $k\in \CC$ with $\mathrm{Re}(k)>1$ and $z\in\HH$. Then for $x\in\RR$, 
\[\sum_{n=-\infty}^{\infty} \bfe(n x)(z+n)^{-k}=\frac{(-2 \pi i)^k}{\Gamma(k)} \sum_{\substack{r \in \mathbb{Z}-x \\ r>0}} r^{k-1} \bfe(r z)\]
where we chose a branch of the logarithm compatible with our choice of $\sqrt{\cdot}$.
\end{lemma}

\begin{acknowledgements}
We thank Jan Hendrik Bruinier, Kaiyuan Gu, Xuanlin Huang, Kai-Wen Lan, and Fei Si for helpful conversations and suggestions, and Xinyi Yuan for his interest in this project. We are also grateful to the two anonymous referees for their insightful comments, which helped to improve the introduction of the paper.

\end{acknowledgements}

\bibliographystyle{plain}
\bibliography{main}

\end{document}